\documentclass[11pt,a4paper]{amsart}

\usepackage[T1]{fontenc}
\usepackage[utf8]{inputenc}
\usepackage{lmodern}
\usepackage{amsmath,amssymb,amsthm}
\usepackage{mathrsfs}
\usepackage[noadjust]{cite}
\usepackage{enumitem}
\usepackage{xcolor}
\IfFileExists{microtype.sty}{\usepackage{microtype}}{}
\usepackage[
  colorlinks=true,
  linkcolor=red,
  citecolor=blue,
  urlcolor=blue
]{hyperref}

\allowdisplaybreaks[4]

\numberwithin{equation}{section}

\setlist[itemize]{topsep=4pt,itemsep=2pt,parsep=1pt}
\setlist[enumerate]{topsep=4pt,itemsep=3pt,parsep=1pt}

\theoremstyle{plain}
\newtheorem{theorem}{Theorem}[section]
\newtheorem{proposition}[theorem]{Proposition}
\newtheorem{lemma}[theorem]{Lemma}
\newtheorem{corollary}[theorem]{Corollary}

\theoremstyle{definition}
\newtheorem{hypothesis}[theorem]{Hypothesis}
\newtheorem{definition}[theorem]{Definition}
\newtheorem{notation}[theorem]{Notation}
\newtheorem{convention}[theorem]{Convention}
\newtheorem{recollection}[theorem]{Recollection}

\theoremstyle{remark}
\newtheorem{remark}[theorem]{Remark}

\hypersetup{
  pdftitle={Stratification of Artin Motives over Local Fields},
  pdfauthor={Peng Xu}
}

\begin{document}

\begin{center}
{\large\bfseries STRATIFICATION OF ARTIN MOTIVES OVER LOCAL FIELDS}

\vspace{0.5cm}
Peng Xu
\end{center}

\bigskip
\centerline{\bfseries Abstract}
\medskip

\leftskip10truemm
\rightskip10truemm
\noindent\hspace{1em}%
Let \(k\) be a field of characteristic \(p>0\). We prove that for
every \(p\)-decomposition group \(P\), the tensor-triangulated category
\(\operatorname{DPerm}(P;k)\) is stratified and its Balmer spectrum
is generically noetherian. As an arithmetic application, we show that
for a nonarchimedean local field \(F\) with residue characteristic
\(\ell\), the category \(\operatorname{DAM}(F;k)\) of derived Artin motives is
stratified if and only if \(\ell\neq p\). In the stratified case its Balmer
spectrum is generically noetherian; consequently,
\(\operatorname{DAM}(F;k)\) satisfies the  telescope
conjecture.\\[2mm]

\leftskip0truemm
\rightskip0truemm
\medskip

\hypertarget{introduction}{%
\section{Introduction}\label{introduction}}

\quad Stratification theory provides a support-theoretic mechanism for recovering
the localizing tensor ideals of a large tensor-triangulated category
from the tensor-triangular geometry of its compact subcategory. If \(\mathscr T\) is a
rigidly-compactly generated tensor-triangulated category with spectrum
\(\operatorname{Spc}(\mathscr T^c)\) weakly noetherian, then \(\mathscr T\) is said to be stratified if support induces mutually
inverse bijections
\[
 \left\{
   \text{localizing  ideals of }\mathscr T
 \right\}
 \longleftrightarrow
 \left\{
   \text{subsets of }\operatorname{Spc}(\mathscr T^c)
 \right\}.
\]
This theory has proved effective in many areas, including algebraic
geometry \cite{Krause,Nee92}, modular representation theory
\cite{BIK11,BIKP18,BBIKP25}, and algebraic topology \cite{BarthelHeardSanders}. More recently, it
has been extended to non-noetherian settings \cite{Zou}.

Permutation modules form a natural meeting point of modular
representation theory, Mackey functors, and Artin motives. Fix a field \(k\) of characteristic \(p\). For finite groups, Balmer and Gallauer determined
the tensor-triangular geometry of permutation modules and proved that the big derived category
\(\operatorname{DPerm}(G;k)\) of permutation modules is stratified and satisfies the
telescope conjecture \cite{BalmerGallauerGeometry}.
They subsequently  proved that the big
category of Artin motives over a finite field is stratified and
satisfies the telescope conjecture \cite{BalmerGallauerArtin}. For an arbitrary field \(F\),
the Grothendieck--Galois correspondence provides a tensor-triangulated
equivalence
\[
 \operatorname{DAM}(F;k)
 \simeq
 \operatorname{DPerm}(G_F;k).
\]
It is therefore natural to ask whether the categories of Artin motives
over local and global fields are stratified. Our first main result gives a complete answer for nonarchimedean local
fields:
\medskip
\begin{theorem}\label{thm1.1}
Let \(F\) be a nonarchimedean local field with residue field \(\kappa_F\). Then
\[
 \operatorname{DAM}(F;k)\text{ is stratified}
 \quad\Longleftrightarrow\quad
 \operatorname{char}(\kappa_F)\neq p.
\]
\end{theorem}

The archimedean local fields also lie on the positive side:
\(\operatorname{DAM}(\mathbb R;k)\) and
\(\operatorname{DAM}(\mathbb C;k)\) are stratified by the finite-group
result of Balmer and Gallauer. Thus we actually classify all local fields \(F\) such that  \(\operatorname{DAM}(F;k)\) is stratified. The obstruction used for the wild direction of Theorem~\ref{thm1.1} is not
confined to local fields. If \(K\) is any global field, then
\(H^1_{\mathrm{cts}}(G_K;\mathbb F_p)\) is infinite-dimensional.
Then the argument used in the wild local case shows that
\(\operatorname{DAM}(K;k)\text{ is not stratified}\) for every global field \(K\).

In the case
\(\operatorname{char}(\kappa_F)\neq p\),  we prove in addition that
\(\operatorname{Spc}\bigl(\operatorname{DAM}^{\rm gm}(F;k)\bigr)\)
is generically noetherian. Combining this with the tame part of
Theorem~\ref{thm1.1} yields our second main result:
\medskip
\begin{theorem}\label{thm1.2}
Let \(F\) be a nonarchimedean local field such that
\(\operatorname{char}(\kappa_F)\neq p\). Then
\(\operatorname{DAM}(F;k)\) satisfies the telescope conjecture, i.e., every smashing localizing ideal of
\(\operatorname{DAM}(F;k)\) is compactly generated.
\end{theorem}
\smallskip
The main difficulty is that stratification does not, in general,
pass formally through filtered colimits. Consequently, the
finite-group stratification theorem does not by itself imply either
the local-to-global principle or pointwise minimality for a profinite
group. Moreover, the spectrum may contain points with nonopen closed
pro-\(p\) subgroup parameters, and such points cannot be captured at
any single finite stage.

We overcome this difficulty by separating the two ingredients of
stratification. In the tame case, the relevant \(p\)-Sylow subgroup is
a \(p\)-decomposition group \(P=A\rtimes_{\chi}B,\) where
\(A\simeq B\simeq(\mathbb Z_p,+)\) and \(\chi\colon B\longrightarrow\operatorname{Aut}_{\mathrm{cts}}(A)\simeq \mathbb Z_p^\times\)
is a continuous injective homomorphism. We first classify its closed subgroups and
determine the possible Weyl pro-\(p\) groups and their continuous
cohomology. This proves that \(\operatorname{Spc}\bigl(\operatorname{DPerm}(P;k)^c\bigr)\)
is generically noetherian and that its Hochster dual is weakly
scattered. The latter property yields the local-to-global principle.

Pointwise minimality is proved separately for the four types of closed
subgroups occurring in \(P\). The trivial subgroup is treated using
cohomological localization and the minimality theorem of
Heyer--Schneider \cite{HeyerSchneider}. Open subgroups  are reduced to the finite-group
case via modular fixed points, using the right adjoint to lift
pointwise minimality back to \(P\), while the
vertical subgroups \(p^rA\) are handled by a localized right-adjoint
capture argument. The transverse procyclic subgroups require the
principal continuity argument of the paper. For these points we
introduce stalkwise \(H\)-strata and selected-restriction functors,
recover the compact parts of the relevant strata from finite
quotients, establish compatible normalizer equivalences at every
finite stage. Thus we use filtered continuity only for the compact localized strata where it is
available, without asserting that stratification itself passes to the
profinite limit. Finally, prime-to-\(p\) reduction and finite descent
transfer stratification  from the
\(p\)-decomposition group to the absolute Galois group of the local field.

The paper is organized as follows.
Section~\ref{sec:preliminaries} recalls some basic facts about derived permutation modules and the Balmer spectrum of the associated category
of compact objects.
Section~\ref{sec:wild} proves the obstruction in residue
characteristic \(p\) and derives the corresponding negative result for
global fields. Section~\ref{sec:pdecomp} analyzes the closed subgroups
and Weyl groups of \(p\)-decomposition groups and proves generic
noetherianity and the local-to-global principle.
Section~\ref{sec:cohomological} develops the analytic and
cohomological input needed for pointwise minimality.
Section~\ref{sec:continuity} introduces stalkwise \(H\)-strata,
establishes their compact continuity, and proves the  selected-restriction equivalences.
Section~\ref{sec:minimality} proves  stratification for every
\(p\)-decomposition group. Finally, Section~\ref{sec:descent} passes to absolute Galois groups, proves the
local-field classification, and establishes the telescope conjecture
in the tame case.

\section{Preliminaries}\label{sec:preliminaries}

We begin by recalling some basic facts about derived permutation modules to fix our notation; see \cite{BalmerGallauerPermutation,BalmerGallauerArtin}
for further details. We assume familiarity with the basic notions of tensor-triangular geometry; see
\cite{Balmer,BalmerFavi,BarthelHeardSanders} for background.
\begin{hypothesis}
Throughout, \(G\) is a profinite group and \(k\) is a field of
characteristic \(p>0\).
\end{hypothesis}
\begin{convention}We write \(H\leq_GK\) if a \(G\)-conjugate of \(H\) is contained in \(K\), and write \(H\leq_oG\) if \(H\) is an open subgroup of \(G\). \(\operatorname{Sub}_p(G)\) denotes the set of closed pro-\(p\) subgroups of \(G\) and \(\operatorname{Sub}_p(G)/G\) denotes the quotient of \(\operatorname{Sub}_p(G)\) under \(G\)-conjugation.  We write \(\mathrm{Max}_o(G)\) for the set of maximal proper open subgroups of \(G\). For a spectral space \(X\), \(\operatorname{gen}_X(x)\) denotes the set of generalizations of \(x\).  The Hochster dual of \(X\) is denoted \(X^*\).
\end{convention}
\begin{recollection}
\label{rec2.3}
Let \(\mathcal A_G:=\operatorname{Mod}^{\mathrm{sm}}_k(G)\) be the category of smooth \(kG\)-modules, equivalently discrete
continuous \(kG\)-modules. It is a locally noetherian Grothendieck category with noetherian objects  the finite-dimensional smooth modules. For a discrete continuous \(G\)-set \(X\),
 \(k(X)\) is a free \(k\)-module on \(X\) equipped with the induced \(G\)-action. This yields a functor \(k(-)\) from the category of discrete continuous \(G\)-sets to \(\mathcal A_G\). A \(kG\)-module is called permutation if it is in the essential image of \(k(-)\). We denote by
 \(\operatorname{Perm}(G;k)\)  the category of permutation modules, and by \(\operatorname{perm}(G;k)\) the subcategory of finite
permutation modules.

A complex \(X\in K(\operatorname{Perm}(G;k))\) is called
\(G\)-acyclic if the fixed-point complex \(X^H\) is acyclic for every
open subgroup \(H\leq_oG\). Let \(K_{G\text{-}\mathrm{ac}}\bigl(\operatorname{Perm}(G;k)\bigr)\)
denote the full subcategory of \(G\)-acyclic complexes. A morphism is a
\(G\)-quasi-isomorphism if its cone is \(G\)-acyclic. The derived
category of permutation modules is the Verdier localization
\[
 \operatorname{DPerm}(G;k)
 :=
 K(\operatorname{Perm}(G;k))[\mathrm{QI}_G^{-1}]
 \simeq
 \frac{K(\operatorname{Perm}(G;k))}
      {K_{G\text{-}\mathrm{ac}}
        (\operatorname{Perm}(G;k))}.
\]
The quotient functor restricts to an equivalence
\[
 \operatorname{Loc}_{K(\operatorname{Perm}(G;k))}
 \{k(G/H)\mid H\leq_oG\}
 \xrightarrow{\sim}
 \operatorname{DPerm}(G;k).
\]
Moreover, \(\operatorname{DPerm}(G;k)\) is a rigidly-compactly generated tensor triangulated category, with compact subcategory
\[
 \operatorname{DPerm}(G;k)^c\simeq \operatorname{thick}_{K(\operatorname{Perm}(G;k))} \{k(G/H)\mid H\leq_oG\}\simeq
K^b(\operatorname{perm}(G;k))^\natural .
\]
Here the symmetric monoidal structure is induced by  diagonal \(G\)-action and \((-)^\natural\) denotes idempotent completion. We
write
\(
 \mathscr T(G):=\operatorname{DPerm}(G;k),\)
and
 \(\mathscr K(G):=\mathscr T(G)^c\) for abbreviation.
\end{recollection}

\begin{recollection}\label{rec2.4}
Let \(F\) be a field with absolute Galois group \(G_F\). The classical Grothendieck-Galois correspondence
induces a canonical equivalence of tt-categories
\(\operatorname{DAM}(F;k)\simeq\operatorname{DPerm}(G_F;k)\) with the triangulated category of Artin motives. In particular, after passing to compact parts, the category of geometric Artin motives \(\operatorname{DAM}^{\rm gm}(F;k)\) identifies with \(\mathscr K(G_F)\).
\end{recollection}

\begin{notation}For a profinite group \(G\), \(D^b(kG)\) denotes the bounded derived category of bounded complexes of finite-dimensional discrete continuous \(kG\)-modules. We write \(
 \mathcal V_G=\operatorname{Spc}(D^b(kG))\simeq \operatorname{Spec}^{h}H^\bullet_{\mathrm{cts}}(G;k)\) for the Balmer spectrum of \(D^b(kG)\) and denote its  irrelevant closed point
by \(\mathfrak m_G=H^{>0}(G;k).\)
\end{notation}

\begin{recollection}\label{rec2.6}
For a
profinite group \(G\),  we write \(D(\mathcal A_G)\) for
the unbounded derived category and
\(K(\operatorname{Inj}\mathcal A_G)\) for the homotopy category of
unbounded complexes of injective objects of \(\mathcal A_G\). Realization gives symmetric monoidal exact
functors
\[
 q_G:\mathscr K(G)\longrightarrow D^b(kG),
 \qquad
 \Upsilon_G:\mathscr T(G)\longrightarrow D(\mathcal A_G),
\]
where \(\Upsilon_G\) preserves coproducts and its restriction to
\(\mathscr K(G)\) is the composite of \(q_G\) with
the canonical functor \(D^b(kG)\to D(\mathcal A_G)\).  We write \(\mathscr K_{\mathrm{ac}}(G):=\ker(q_G).\)

For every profinite group \(G\), there is a bijection of underlying sets
\[
 \operatorname{Spc}(\mathscr K(G))=\coprod_{(H)}\mathcal V_{N_{G}(H)/H}
\]
over conjugacy classes \((H)\) of closed pro-\(p\)-subgroups \(H\) of \(G\).
Thus each point \(x\in \operatorname{Spc}(\mathscr K(G))\) is of the form
\(\mathcal P_G(H,\mathfrak q)\) with \(\mathfrak q\in\mathcal V_{N_G(H)/H}\), uniquely up to \(G\)-conjugacy. We call \([H]_{G}\) a subgroup parameter of \(x\).
This defines a subgroup-parameter map
\[
 \sigma_G:\operatorname{Spc}(\mathscr K(G))
       \longrightarrow \operatorname{Sub}_p(G)/G,
 \qquad
 \sigma_G\bigl(\mathcal P_G(H,\mathfrak q)\bigr)=[H]_G,
\]
where \(\operatorname{Sub}_p(G)\) denotes the set of closed pro-\(p\) subgroups. Sometimes we
write \(X_G:=\operatorname{Spc}(\mathscr K(G))\) for abbreviation.
The following  facts will be used separately.  If \(y=\mathcal P_G(H',\mathfrak q')\) is a generalization of
\(x=\mathcal P_G(H,\mathfrak q)\), then \(x\subseteq y\) as prime ideals and \( H'\leq_G H\). For every open
subgroup \(K\leq_oG\), we have
\[
 \operatorname{kos}_G(K)\in\mathcal P_G(H,\mathfrak q)
       \Longleftrightarrow H\leq_GK,
 \qquad
 \mathcal P_G(H,\mathfrak q)\in
       \operatorname{supp}(\operatorname{kos}_G(K))
       \Longleftrightarrow H\nleq_GK,
\]
where \(\operatorname{kos}_G(K):=
 \mathop{{}^{\otimes}\!\operatorname{Ind}}\nolimits_K^G
 \bigl(0\rightarrow k\xrightarrow{\,1\,}k\rightarrow0\bigr)
 \in\mathscr K(G).\)
\end{recollection}

\section{The wild case}\label{sec:wild}

In this section, we prove the negative direction of the local-field classification. We first show that a \(p\)-Sylow subgroup of the
absolute Galois group admits an infinite elementary abelian pro-\(p\) quotient. We then use a cohomological point that is not weakly visible
to produce a non-weakly-visible point in the corresponding permutation spectrum. The same argument shows that the category of Artin motives
over any global field is not stratified.
\begin{lemma}\label{lem3.1}
Let \(f:X\to Y\) be a spectral map. If \(f^{-1}(y)=\{x\}\) and \(y\) is
weakly visible, then \(x\) is weakly visible.
\end{lemma}
\begin{proof}
Write \(\{y\}=V\cap W^c\) with \(V,W\) Thomason. Spectral maps pull
Thomason subsets back to Thomason subsets, and therefore
\(
 \{x\}=f^{-1}(V)\cap f^{-1}(W)^c.
\)
\end{proof}

\begin{lemma}\label{lem3.2}

Let \(F\) be a nonarchimedean local field with residue characteristic \(p\). Then every
\(p\)-Sylow subgroup \(P\leq G_F\) has a closed normal subgroup
\(H\trianglelefteq P\) such that \(P/H\simeq \prod_{n\geq1}C_p.\)
\end{lemma}
\begin{proof}
First suppose \(\operatorname{char}(F)=p\). The Artin--Schreier sequence
and additive Hilbert 90 give
\(
 H^1(G_F,\mathbb F_p)\simeq F/(a^p-a).
\) If \(\pi\) is a uniformizer, the classes \(\pi^{-n}\) for positive \(n\)
not divisible by \(p\) are linearly independent. Indeed a nonzero finite
linear combination has negative valuation not divisible by \(p\),
whereas for \(v(a)<0\) one has \(v(a^p-a)=pv(a)\) and for \(v(a)\geq0\)
one has \(v(a^p-a)\geq0\).

Restriction \(H^1(G_F,\mathbb F_p)\to H^1(P,\mathbb F_p)\) is injective:
a nonzero character \(G_F\to C_p\) has a quotient of order \(p\), and
the image of a \(p\)-Sylow in that quotient must be all of it. We can
therefore choose countably many independent continuous characters
\(P\to C_p\). Their product map \(P\to \prod_{n\geq1}C_p\) is surjective, since its image is dense by linear independence and
closed by compactness.

If \(\operatorname{char}(F)=0\), then \(F/\mathbb Q_p\) is finite. By
 \cite[Theorem~5]{Labute},  \(P\) has rank \(\aleph_0\). Hence its Frattini
quotient \(P/\Phi(P)\) is the countably infinite elementary abelian
pro-\(p\) group \(\prod_{n\geq1}C_p\). Then we can take \(H=\Phi(P)\).
\end{proof}

\begin{lemma}
\label{lem3.3}
Let \(W\) be a profinite group and let \(E\leq W\) be a closed
subgroup. Restriction induces a continuous map \(r\colon \mathcal{V}_E\longrightarrow \mathcal{V}_W\)
such that \( r(\mathfrak m_E)=\mathfrak m_W,\) \(r^{-1}(\mathfrak m_W)=\{\mathfrak m_E\}.\)
\end{lemma}

\begin{proof}
Let
\(\rho\colon H^\bullet_{\mathrm{cts}}(W;k)\longrightarrow H^\bullet_{\mathrm{cts}}(E;k)\) be the restriction homomorphism. Since \(\rho\) preserves degrees and
is the identity in degree zero,
\(\rho^{-1}(\mathfrak m_E)=\mathfrak m_W\). This proves \(r(\mathfrak m_E)=\mathfrak m_W\).

The fiber statement is a standard fact for finite groups. Now let \(W\) be profinite. We have \(\mathcal{V}_W\simeq\varprojlim_U \mathcal{V}_{W/U}\) and \(\mathcal{V}_E\simeq\varprojlim_U \mathcal{V}_{EU/U},\)
and \(r\) is the inverse limit of the finite restriction maps \(r_U\colon \mathcal{V}_{EU/U}\to \mathcal{V}_{W/U}\). Suppose \(\mathfrak q\in \mathcal{V}_E\) satisfies
\(r(\mathfrak q)=\mathfrak m_W\) with \(\mathfrak q_U\in \mathcal{V}_{EU/U}\) its coordinate. Then \(r_U(\mathfrak q_U)=\mathfrak m_{W/U}\). By the finite-group case,
\(\mathfrak q_U=\mathfrak m_{EU/U}\) for every \(U\). The inverse-limit description therefore gives \(\mathfrak q=\mathfrak m_E\). Hence
\(r^{-1}(\mathfrak m_W)=\{\mathfrak m_E\}.\)
\end{proof}

\begin{proposition}\label{prop3.4}

If \(F\) is a nonarchimedean local field with residue characteristic \(p\), then \(\operatorname{DAM}(F;k)\) is not stratified.

\end{proposition}\begin{proof}

Choose \(P,H\) as in Lemma \ref{lem3.2} and write \(E:=P/H\simeq\prod_{n\geq1}C_p\). Since
\(H\trianglelefteq P\), we have \(P\leq N_{G_F}(H)\), and the induced injection
\[E\longrightarrow W:=N_{G_F}(H)/H\] has compact, and hence closed image. Thus \(E\) is a closed subgroup
of \(W\). Consider
\[\mathcal V_E\xrightarrow{r}\mathcal V_W\xrightarrow{\check\psi_H}X_F\]
where the second map is an injective spectral map by Recollection~\ref{rec2.6}. It follows from Lemma \ref{lem3.3}
 that the fiber over \( x=\check\psi_H(\mathfrak m_W)\)
is exactly \(\{\mathfrak m_E\}\). If \(X_{G_F}\) is weakly noetherian,
then \(x\) would be weakly visible and Lemma
\ref{lem3.1} would make \(\mathfrak m_E\) weakly
visible. This contradicts \cite[Example~6.10]{BalmerGallauerArtin}. Thus
\(X_{G_F}\) is not weakly noetherian, and so \(\operatorname{DAM}(F;k)\) is not stratified by \cite[Theorem~A]{Zou}.
\end{proof}

\begin{remark}
The same argument shows that \(\operatorname{DAM}(F;k)\) is not stratified for every global field \(F\). Indeed, global class field
theory and Artin--Schreier theory give
\(\dim_{\mathbb F_p}H^1_{\mathrm{cts}}(G_F;\mathbb F_p)=\infty\), so a
\(p\)-Sylow subgroup of \(G_F\) surjects onto \(\prod_{n\geq1}C_p\) and then
the proof of Proposition~\ref{prop3.4} applies verbatim.
\end{remark}

\section{\texorpdfstring{\(p\)}{p}-decomposition groups and their spectra}\label{sec:pdecomp}

We now begin the proof of the tame direction of the classification. Its basic group-theoretic model is a \(p\)-decomposition group
\(P=A\rtimes_{\chi}B\), which will arise as a \(p\)-Sylow subgroup controlling the  tensor-triangular geometry of the absolute
Galois group in the tame residue-characteristic case. In this section we analyze the closed subgroups of \(P\), determine the possible Weyl
pro-\(p\) groups and their reduced continuous cohomology. Then we prove that \(X_P\) is generically noetherian and
Hochster weakly scattered. Consequently, \(\mathscr T(P)\) satisfies the local-to-global principle.

\begin{definition}
\label{def4.1}
 A \emph{ \(p\)-decomposition group} is a pro-\(p\) group equipped with a decomposition \(P=A\rtimes_{\chi}B,\) where
\(A\simeq B\simeq(\mathbb Z_p,+)\) and \(\chi\colon B\longrightarrow\operatorname{Aut}_{\mathrm{cts}}(A)\simeq \mathbb Z_p^\times\)
is a continuous injective homomorphism.
\end{definition}
\medskip

We fix additive coordinates on \(A\) and \(B\), and write the group law as \((a,b)(a',b')=\bigl(a+\chi(b)a',\,b+b'\bigr).\)
The canonical projection is denoted by \(\pi\colon P\longrightarrow B\) with \(\pi(a,b)=b,\) so that \(A=\ker(\pi)\).

\begin{remark}
Our definition of a \(p\)-decomposition group is a slight modification of Ivanov's definition \cite[Definition~2.1]{Ivanov}: we additionally
require the action \(\chi\) to be faithful. This excludes only the nonfaithful inversion action at \(p=2\).
\end{remark}
\begin{lemma}\label{lem4.3}
Let \(P=A\rtimes_\chi B\) be a \(p\)-decomposition group. For every closed subgroup \(H\leq P\), the Weyl pro-\(p\)
group \(W_P(H)=N_P(H)/H\) is one of the following:
\[
\begin{array}{c|c}
H & W_P(H)\\ \hline
1 & P,\\
H\text{ open} & \text{a finite }p\text{-group},\\
H=p^rA\ (r\geq0) & (C_{p^r})\rtimes\mathbb Z_p,\\
H\not\leq A\text{ procyclic} & \text{a finite cyclic }p\text{-group}.
\end{array}
\]
Moreover, the \(P\)-conjugacy classes of nonopen closed subgroups form a countable set.
\end{lemma}
\begin{proof}
We first show that every nontrivial nonopen closed subgroup of \(P\) is procyclic. Put \(L=H\cap A\) and let \(\pi:P\to B\) be projection. Both
\(L\) and \(\pi(H)\) are closed subgroups of \(\mathbb Z_p\), hence are either trivial or open. If both \(L\) and \(\pi(H)\) are nontrivial, then they have finite index in
\(A\) and \(B\), respectively. Since \(AH=\pi^{-1}\bigl(\pi(H)\bigr),\) the index formula gives
\[
    [P:H]
      =[P:AH][AH:H]
      =[B:\pi(H)][A:L]
      <\infty.
\]
As \(H\) is closed, it is therefore open in \(P\). If \(L=1\), projection identifies \(H\) with the closed subgroup
\(\pi(H)\) of \(B\), so \(H\) is trivial or procyclic. If \(\pi(H)=1\), then \(H\leq A\) and is again procyclic. Thus every nontrivial nonopen
\(H\) is procyclic, for every prime \(p\).

We next identify the normal procyclic subgroups of open subgroups of
\(P\). Let \(U\leq_o P\). Since \(U\cap A=\ker(\pi|_U),\)
the subgroup \(U\cap A\) is normal in \(U\). Moreover, since \(U\) is open in \(P\), the subgroup \(U\cap A\) is open in
\(A\simeq\mathbb Z_p\), and is therefore procyclic.

Let \(C\trianglelefteq U\) be a procyclic subgroup. We claim that \(C\leq A\). Suppose, to the contrary, that \(C\not\leq A\). We first
show that \(C\cap A=1.\) Indeed, if \(C\cap A\neq1\), then \(C\cap A\) is a nontrivial closed subgroup of the procyclic pro-\(p\) group \(C\), and hence has finite
index in \(C\). Since \(\ker(\pi|_C)=C\cap A,\) we have
\(\pi(C)\simeq C/(C\cap A),\) so \(\pi(C)\) is finite. But \(\pi(C)\leq B\simeq\mathbb Z_p\), and
\(\mathbb Z_p\) has no nontrivial finite subgroup. Thus \(\pi(C)=1\), which implies \(C\leq A\), a contradiction. Hence \(C\cap A=1\).

Since \(C\trianglelefteq U\) and \(U\cap A\trianglelefteq U\), we have \([C,U\cap A]\leq C\cap(U\cap A)=C\cap A=1.\)
Thus \(C\) centralizes \(U\cap A\). Since \(C\not\leq A=\ker(\pi)\), choose \(c\in C\) such that \(b:=\pi(c)\neq0.\)
Also choose \(0\neq a\in U\cap A\). Conjugation by \(c\) on \(A\) is multiplication by \(\chi(b)\), and hence
\(a=cac^{-1}=\chi(b)a.\) It follows that \((\chi(b)-1)a=0.\) Since \(A\simeq\mathbb Z_p\) and \(a\neq0\), we obtain
\(\chi(b)=1\). The faithfulness of \(\chi\) then gives \(b=0\), contradicting the choice of \(b\). Therefore every normal procyclic subgroup \(C\trianglelefteq U\) is
contained in \(A\), and hence in \(U\cap A\). Since \(U\cap A\) is itself normal and procyclic, it is the unique largest normal procyclic subgroup of \(U\).

If \(0\neq H\leq A\), then there is a unique \(r\geq0\) such that \(H=p^rA.\) Conjugation by an element \((a,b)\in P\) acts on \(A\) as
multiplication by the unit \(\chi(b)\), and hence preserves \(p^rA\). Thus \(p^rA\trianglelefteq P\), so
\(W_P(p^rA)=P/p^rA\simeq (A/p^rA)\rtimes_{\bar\chi}B\simeq C_{p^r}\rtimes_{\bar\chi}\mathbb Z_p.\)

Suppose now that \(H\) is procyclic and not contained in \(A\). Then \(H\cap A=1\). Indeed, otherwise \(H\cap A\) would have finite index
in \(H\), so \(\pi(H)\simeq H/(H\cap A)\) would be a finite subgroup of \(B\simeq\mathbb Z_p\), hence trivial.
This would imply \(H\leq A\), a contradiction. Therefore the restriction \(\pi|_H:H\longrightarrow B\) is injective.

If \(g\in N_P(H)\), then \(ghg^{-1}\in H\) for every \(h\in H\). Since \(B\) is abelian, \(\pi(ghg^{-1})=\pi(h).\)
The injectivity of \(\pi|_H\) gives \(ghg^{-1}=h\). Thus \(g\) centralizes \(H\), and consequently \(N_P(H)=C_P(H).\) The group \(C_P(H)\) cannot be open in \(P\). Indeed, if it were open, then \(H\), being central in \(C_P(H)\), would be a normal procyclic subgroup of an open subgroup of \(P\). The preceding paragraph would
then imply \(H\leq C_P(H)\cap A\leq A,\) contrary to \(H\not\leq A\).

Now \(C_P(H)\) is a nontrivial closed nonopen subgroup of \(P\): it contains \(H\), and centralizers are closed. By the first paragraph,
\(C_P(H)\) is therefore procyclic. Since \(H\) is a nontrivial closed subgroup of the procyclic pro-\(p\) group \(C_P(H)\), it has finite
index. Hence \(W_P(H)=C_P(H)/H\) is a finite cyclic \(p\)-group.

For countability, the nontrivial closed subgroups contained in \(A\) form the countable chain
\[
    A\geq pA\geq p^2A\geq\cdots.
\]
Now let \(H\) be a procyclic subgroup not contained in \(A\). Since \(\pi|_H\) is injective, its image is a nontrivial closed subgroup of
\(B\simeq\mathbb Z_p\), and hence \(\pi(H)=p^sB\) for a unique \(s\geq0\). The unique element of \(H\) projecting to
\(p^s\) is a topological generator of the form \((a,p^s)\), for some \(a\in A\). For \((c,d)\in P\), direct calculation gives
\[(c,d)(a,p^s)(c,d)^{-1}=\bigl(\chi(d)a+(1-\chi(p^s))c,\ p^s\bigr).\]
Put \(I_s=(1-\chi(p^s))A.\) Thus, for fixed \(s\), the conjugacy classes of such subgroups are parametrized by the orbit set \((A/I_s)/\chi(B),\) where \(\chi(B)\) acts by multiplication. Since \(p^s\neq0\) and \(\chi\) is faithful, we have \(\chi(p^s)\neq1\). Hence \(I_s\) is a nonzero ideal of \(A\simeq\mathbb Z_p\), and is therefore open. It
follows that \(A/I_s\) is finite, so there are only finitely many conjugacy classes for each \(s\). Since \(s\) ranges over \(\mathbb N\), the conjugacy classes of procyclic subgroups not contained in \(A\) form a countable set. Together with the chain \(p^rA\) and the trivial subgroup, this proves that the conjugacy
classes of nonopen closed subgroups are countable.
\end{proof}

\begin{proposition}\label{prop4.4}
Let \(P=A\rtimes_\chi B\) be a \(p\)-decomposition group and \(H\leq P\) a nonopen closed subgroup. Then for \(W=N_P(H)/H\), we have
\[
 H^\bullet_{\mathrm{cts}}(W;k)/\sqrt{0}\simeq
 \begin{cases}
  k,
    & W\simeq P,\ \mathbb Z_p,\ \text{or }1,\\[1mm]
  k[u],
    & W\simeq (C_{p^r})\rtimes\mathbb Z_p,\quad r\geq1,\\[1mm]
  k[u],
    & W\text{ is a nontrivial finite cyclic }p\text{-group},
 \end{cases}
\]
where \(u\) is homogeneous of positive degree. In particular, \(\left|\mathcal V_W\right|\leq2,\) and the locus of points of
\(X_P\) having subgroup parameter \([H]_P\) has at most two points.
\end{proposition}

\begin{proof}
If \(W=P\), the Lyndon--Hochschild--Serre spectral sequence for
\[1\longrightarrow A\longrightarrow P\longrightarrow B\longrightarrow1\] is concentrated in total degrees at most two, since
\(\operatorname{cd}_p(A)=\operatorname{cd}_p(B)=1\). Thus every positive-degree element is nilpotent, and the reduced cohomology ring
is \(k\). The cases \(W\simeq\mathbb Z_p\) and \(W=1\) are immediate.

Suppose \(W\simeq C_{p^r}\rtimes\mathbb Z_p,\) \( r\geq1\). The continuous Lyndon--Hochschild--Serre spectral sequence
\[
 E_2^{i,j}
 =
 H^i_{\mathrm{cts}}\!\left(
   \mathbb Z_p;H^j(C_{p^r};k)
 \right)
 \Longrightarrow
 H^{i+j}_{\mathrm{cts}}(W;k)
\]
has only the columns \(i=0,1\). The action of \(\mathbb Z_p\) on \(H^\bullet(C_{p^r};k)\) is trivial: the corresponding automorphisms
of \(C_{p^r}\) are congruent to the identity modulo \(p\), and hence act trivially on the standard degree-one and degree-two generators.
It follows that the restriction map
\[
 H^\bullet_{\mathrm{cts}}(W;k)
 \longrightarrow H^\bullet(C_{p^r};k)
\]
is surjective. Its kernel is the first filtration piece and therefore has square zero. Consequently it induces an isomorphism on reduced rings.

The standard periodic resolution of \(C_{p^r}\) gives \(H^\bullet(C_{p^r};k)/\sqrt0\simeq k[u].\)
The same calculation proves the assertion for a nontrivial finite cyclic \(p\)-group. Finally, \(\operatorname{Spec}^h(k)\) has one point and
\(\operatorname{Spec}^h(k[u])\) has the two homogeneous primes \((0)\) and \((u)\). The final assertion follows from  \cite[Corollary~3.10]{BalmerGallauerArtin}.
\end{proof}

\begin{proposition}
\label{prop4.5}
For a \(p\)-decomposition group \(P\), the spectral space \(X_P\) is generically noetherian. In particular, \(X_P\) is weakly noetherian.
\end{proposition}
\begin{proof}
Let \(x=\mathcal P_P(H,\mathfrak q)\) and \(z=\mathcal P_P(L,\mathfrak r)\) a generalization of \(x\). We
first show that, after conjugating \(L\) we would have \(L\leq H\) and \( [H:L]\leq p^2.\) Choose a descending cofinal sequence
\((N_n)\) of open normal subgroups of \(P\), and put
\[
 Q_n=P/N_n,\qquad H_n=HN_n/N_n,\qquad L_n=LN_n/N_n.
\]
Since \(Q_n\) is metacyclic (i.e., cyclic-by-cyclic),  any subquotient of \(Q_n\)
isomorphic to \((C_p)^r\) has \(r\leq2\).

Since the images \(z_n,x_n\in X_{Q_n}\)
satisfy \(x_n\in\overline{\{z_n\}}\), by \cite[Theorem~11.10, Remark~11.11]{BalmerGallauerGeometry}, there is an elementary abelian
section \(E_n\) of \(Q_n\) and points
\(\widetilde z_n=\mathcal P_{E_n}(J_z,\mathfrak a_z),\) \(\widetilde x_n=\mathcal P_{E_n}(J_x,\mathfrak a_x)\)
mapping to \(z_n\) and \(x_n\), respectively, with
\(\widetilde x_n\in\overline{\{\widetilde z_n\}}\).  Hence
\cite[Corollary~7.13]{BalmerGallauerGeometry} gives
\[
        J_z\leq J_x,\qquad [J_x:J_z]\leq p^2.
\]
Write \(E_n=S_n/T_n\), and let
\(\rho_n:S_n\twoheadrightarrow E_n\) be the quotient map.
By \cite[Remark~7.6(b) and Proposition~7.11]
{BalmerGallauerGeometry}, the subgroup parameters of the images of
\(\widetilde z_n\) and \(\widetilde x_n\) in
\(\operatorname{Spc}(\mathscr K(Q_n))\) are
\([\rho_n^{-1}(J_z)]_{Q_n}\) and \([\rho_n^{-1}(J_x)]_{Q_n}\), respectively.
Since these images are \(z_n\) and \(x_n\), there exist
\(a_n,b_n\in Q_n\) such that
\(\rho_n^{-1}(J_z)=L_n^{a_n},\) \(\rho_n^{-1}(J_x)=H_n^{b_n}.\) Since \(J_z\leq J_x\),  there exists \(g_n\in P\) such that
\(
 L^{g_n}N_n\leq HN_n
\)
and \([HN_n:L^{g_n}N_n]\leq p^2.\)

Let \(C_n\subseteq P\) be the set of all \(g\) satisfying these two
conditions. The sets \(C_n\) are nonempty and clopen, and
\(C_{n+1}\subseteq C_n\). This gives
\(g\in\bigcap_nC_n\) by compactness. Then by cofinality, we have \(L^g\leq H\).
If \([H:L^g]>p^2\), then \(p^2+1\) distinct cosets remain distinct in
a sufficiently large finite-quotient, contradicting the property of \(C_n\). This proves the claim.

Since a  topologically finitely generated profinite group has only finitely
many open subgroups of bounded index, there exist finitely many
\(P\)-conjugacy classes
\([K_1]_P,\ldots,[K_m]_P\) such that
\(
 \operatorname{gen}_{X_P}(x)
 \subseteq
 \bigcup_{i=1}^m X_{[K_i]},
\)
where \(X_{[K]}=\{\mathcal P_P(K',\mathfrak s)\mid K'\sim_P K\}\) is the continuous image of
\(\mathcal V_{N_P(K)/K}\). By Lemma~\ref{lem4.3} and Proposition~\ref{prop4.4}, each \(X_{[K_i]}\) is
noetherian. Thus \(\operatorname{gen}_{X_P}(x)\) is also noetherian, which proves that \(X_P\) is generically noetherian. The last assertion follows from
\cite[Lemma~9.9]{BarthelHeardSanders}.
\end{proof}

Recall from \cite[Definitions~2.14 and~2.15]{Zou} that for a topological space \(X\) and a subset $S\subseteq X$, a point $x$ is called a weakly isolated point of $S$ if there exists an open subset $U$ of $X$ such that $\{x\}\subseteq U\cap S\subseteq \overline{\{x\}}$. A point $x$ is said to be an isolated point if there exists an open subset $U$ of $X$ such that $\{x\}= U\cap S$.      The space $X$ is said to be (weakly) scattered if every nonempty closed subset of $X$ has a (weakly) isolated point. A spectral space is said to be Hochster (weakly) scattered if its Hochster dual is (weakly) scattered.

\begin{lemma}\label{lem4.6}
Every countable spectral space is weakly scattered.
\end{lemma}
\begin{proof}
Let \(C=\{x_1,x_2,\ldots\}\) be a nonempty closed subspace. A closed subspace of a spectral space is spectral. Suppose \(C\) has no weakly
isolated point. Then every \(\overline{\{x_n\}}\) has empty interior in \(C\): if a nonempty open \(U\) were contained in this closure, every
open neighborhood of a specialization of \(x_n\) would contain its generalization \(x_n\), so \(x_n\in U\) and \(x_n\) would be weakly isolated.

Set \(Q_1=C\) and choose recursively nonempty quasi-compact opens \(Q_{n+1}\subseteq Q_n\setminus\overline{\{x_n\}}.\) Quasi-compact opens are clopen in the compact Hausdorff constructible topology. Nested compactness produces \(y\in\bigcap_nQ_n\). Write \(y=x_m\). Then \(y\in Q_{m+1}\), contrary to the choice
\(Q_{m+1}\cap\overline{\{x_m\}}=\varnothing\). Thus every nonempty closed subspace has a weakly isolated point.
\end{proof}

\begin{theorem}\label{thm4.7}
For a \(p\)-decomposition group \(P\), the spectrum \(X_P\) is Hochster weakly scattered. In particular, \(\mathscr T(P)\) satisfies the local-to-global principle.
\end{theorem}
\begin{proof}
For every open normal \(N\trianglelefteq_oP\), its Frattini subgroup \(\Phi(N)\) is an open characteristic subgroup of \(N\), and thus normal in \(P\). Let
\[
 \pi_N:=
 \operatorname{Spc}\!\bigl(
   \operatorname{Infl}_{P/\Phi(N)}^P
 \bigr):
 X_P\longrightarrow
 \operatorname{Spc}\bigl(\mathscr K(P/\Phi(N))\bigr).
\]
Let \(Y_N\) be the image of \(\psi_N\) for the ambient group \(P\), and let \(Z_N\) be the image of \(\psi_{N/\Phi(N)}\) for the ambient
group \(P/\Phi(N)\). By \cite[Lemma~3.18]{BalmerGallauerArtin}, both are closed subsets of their respective spectra. We claim that
\(Y_N=\pi_N^{-1}(Z_N).\) Indeed, let \(z=\mathcal P_P(H,\mathfrak q)\in X_P\). The subgroup parameter of \(\pi_N(z)\) is
\(\left[H\Phi(N)/\Phi(N)\right]_{P/\Phi(N)}.\) Since \(N\trianglelefteq P\), the description of the images of the modular fixed-point maps gives
\[z\in Y_N\quad\Longleftrightarrow\quad N\leq H,\]
whereas
\[
 \begin{aligned}
 \pi_N(z)\in Z_N
 &\quad\Longleftrightarrow\quad
 N/\Phi(N)\leq H\Phi(N)/\Phi(N)\\
 &\quad\Longleftrightarrow\quad
 N\leq H\Phi(N).
 \end{aligned}
\]
It remains to observe that \(N\leq H\Phi(N)\) is equivalent to \(N\leq H\). One implication is immediate. Conversely, suppose that \(N\leq H\Phi(N)\). Since \(\Phi(N)\leq N\), we have \(N=N\cap H\Phi(N)=(N\cap H)\Phi(N).\) If \(N\cap H<N\), choose a maximal proper open subgroup \(M<N\) containing \(N\cap H\). Since \(\Phi(N)\leq M\), we would have
\(N=(N\cap H)\Phi(N)\leq M<N,\) a contradiction. Hence \(N\cap H=N\), and therefore \(N\leq H\). This proves \(Y_N=\pi_N^{-1}(Z_N).\)

The spectrum \(\operatorname{Spc}(\mathscr K(P/\Phi(N)))\) is noetherian by \cite[Proposition~9.1]{BalmerGallauerGeometry}, so the complement of
\(Z_N\) is quasi-compact open. Since \(\pi_N\) is spectral and \(Y_N=\pi_N^{-1}(Z_N),\) the subset \(Y_N\) is closed in \(X_P\) with quasi-compact open
complement. Hence \(Y_N\) is open in \(X_P^*\), and, since it is closed in \(X_P\), its induced topology is \((Y_N)^*\). Moreover, \cite[Lemma~3.18]{BalmerGallauerArtin} identifies \(Y_N\) with \(X_{P/N}.\) Thus \cite[Lemma~7.17(1)]{Sanders} implies that \(Y_N\) is scattered as a subspace of \(X_P^*\).
Set \(O:=\bigcup_{N\trianglelefteq_oP}Y_N.\)
Since \(\mathcal P_P(H,\mathfrak q)\in Y_N\quad\Longleftrightarrow\quad N\leq H,\) and every open subgroup of \(P\) contains an open normal subgroup,
\(O\) is exactly the locus of points with open subgroup parameter. Moreover, the space \(O\) is scattered. Indeed, let \(C\subseteq O\) be
nonempty and closed. Choose \(N\trianglelefteq_oP\) such that \(C\cap Y_N\neq\varnothing\). Then \(C\cap Y_N\) is closed in the
scattered space \(Y_N\), so there exist \(x\in C\cap Y_N\) and an open subset \(V\subseteq Y_N\) such that
\(V\cap C=\{x\}.\) Since \(Y_N\) is open in \(O\), the subset \(V\) is also open in
\(O\). Thus \(x\) is isolated in \(C\), proving that \(O\) is scattered.

The complement \(D=X_P^*\setminus O\) is a closed spectral subspace.
Since we have
\[
 D=
 \coprod_{\substack{[H]_P\\ H\ \mathrm{nonopen}}}
 \mathcal V_{W_P(H)},
\]
it follows from Lemma~\ref{lem4.3} and Proposition~\ref{prop4.4} that \(D\) is countable. Hence
Lemma~\ref{lem4.6} shows that \(D\) is weakly scattered. If a nonempty closed subset of \(X_P^*\) meets
\(O\), use an isolated point of its intersection with \(O\); otherwise it is contained in \(D\), where it has a weakly isolated point.
Therefore \(X_P^*\) is weakly scattered, and \cite[Theorem~7.6]{Zou} implies that \(\mathscr T(P)\) satisfies the local-to-global principle.
\end{proof}

\section{Analytic minimality}
\label{sec:cohomological}

In this section we recall the minimality theorem of Heyer--Schneider and develop the homological input for pointwise minimality. For later use, if \(V\) is a \(k\)-vector space, set
\[
 C(G,V):=\{f:G\to V\mid f\text{ is locally constant}\},
 \qquad (g\cdot f)(x)=f(g^{-1}x).
\]
This is the smooth coinduction of \(V\) from the trivial subgroup.  If
\(M\in\mathcal A_G\), then \(M_{\mathrm{triv}}\) denotes its underlying
vector space equipped with trivial \(G\)-action.

The following result is due to Heyer and Schneider, see \cite[Theorem~1]{HeyerSchneider}.

\begin{theorem}\label{thm5.1}
Let \(G\) be a compact \(p\)-adic Lie pro-\(p\) group.  Then
\(D(\mathcal A_G)\) has no nonzero proper localizing subcategory.
\end{theorem}

\begin{proposition}
\label{prop5.2}
Let \(G\) be a profinite group. Then \(q_G\) induces a tensor-triangulated equivalence
\[
   \frac{\mathscr K(G)}
        {\mathscr K_{\mathrm{ac}}(G)}
 \xrightarrow{\ \sim\ }D^b(kG).
\]
\end{proposition}
\begin{proof}
For every finite quotient \(Q=G/N\), \cite[Theorem~5.13]{BalmerGallauerFinite}  gives
\[
 \frac{\mathscr K(Q)}
      {\mathscr K_{\mathrm{ac}}(Q)}
 \xrightarrow{\ \sim\ }D^b(kQ).
\]

Since a bounded complex of finite-dimensional smooth \(G\)-modules factors
through a finite quotient of \(G\), this shows essential surjectivity.

For fullness, represent a morphism in \(D^b(kG)\) by a roof. The
objects and arrows occurring in this roof form a finite diagram and
therefore factor through a  finite quotient.
The finite-group equivalence lifts the resulting morphism to the
corresponding Verdier quotient.

For faithfulness, represent a morphism in
\(
 \mathscr K(G)/\mathscr K_{\mathrm{ac}}(G)
\)
by a roof
\(
 X\xleftarrow{s}Z\xrightarrow{f}Y
\)
at some finite stage. If its image vanishes in \(D^b(kG)\), the
calculus of fractions provides a quasi-isomorphism \(u:Z'\to Z\)
refining the roof such that \(fu\) is null-homotopic. All these data
therefore factor through one deeper finite quotient. The finite-group
equivalence then implies that the original roof is zero. This concludes the proof.
\end{proof}

\begin{corollary}
\label{cor5.3}
For every profinite group \(G\), we have
\[
 \mathcal V_G\simeq
 \operatorname{Spc}(\mathscr K(G))\setminus
 \bigcup_{c\in\mathscr K_{\mathrm{ac}}(G)}
       \operatorname{supp}(c).
\]
\end{corollary}

\begin{proof}
This follows from  Proposition~\ref{prop5.2} and \cite[Proposition~3.11]{Balmer} immediately.
\end{proof}

\begin{proposition}
\label{prop5.4}

Let \(G\) be a profinite group and suppose that
\(
 \operatorname{gldim}(\mathcal A_G)<\infty.
\)
Then the realization functor \(\Upsilon_G:\mathscr T(G)\longrightarrow D(\mathcal A_G)\)
is a finite localization, and there are equivalences
\[
 \frac{\mathscr T(G)}
      {\operatorname{Loc}^{\otimes}
       \bigl(\mathscr K_{\mathrm{ac}}(G)\bigr)}
 \xrightarrow{\ \sim\ }
 D(\mathcal A_G)
 \xleftarrow{\ \sim\ }
 K(\operatorname{Inj}\mathcal A_G).
\]
\end{proposition}

\begin{proof}
Let
\(
 d=\operatorname{gldim}(\mathcal A_G)<\infty.
\)
By \cite[Proposition~2.3]{Krause},
\(K(\operatorname{Inj}\mathcal A_G)\) is compactly generated and \( K(\operatorname{Inj}\mathcal A_G)^c\simeq D^b(kG).\) Since
\(
 \operatorname{gldim}(\mathcal A_G)<\infty,
\)
every object of \(\mathcal A_G\) has finite injective dimension.
Hence \cite[Example~3.10]{Krause} gives a canonical equivalence
\[
 K(\operatorname{Inj}\mathcal A_G)
 \xrightarrow{\ \sim\ }
 D(\mathcal A_G).
\]
Consequently \(D(\mathcal A_G)\) is compactly generated and \(D(\mathcal A_G)^c\simeq D^b(kG).\)

Let
\(
 \mathcal L
 :=
 \operatorname{Loc}
 \bigl(\mathscr K_{\mathrm{ac}}(G)\bigr).
\)
 For
\(c\in\mathscr K_{\mathrm{ac}}(G)\), the objects \(t\) satisfying
\(
 c\otimes t\in\mathcal L
\)
form a localizing subcategory of \(\mathscr T(G)\). It contains every
compact object, since \(\mathscr K_{\mathrm{ac}}(G)\) is a thick ideal.
Since \(\mathscr T(G)\) is compactly generated, this localizing
subcategory is all of \(\mathscr T(G)\). Now fix \(t\in\mathscr T(G)\). The full subcategory
\[
 \{x\in\mathscr T(G)\mid x\otimes t\in\mathcal L\}
\]
is localizing and, by the preceding argument, contains
\(\mathscr K_{\mathrm{ac}}(G)\). It therefore contains
\(\mathcal L=\operatorname{Loc}(\mathscr K_{\mathrm{ac}}(G))\).
Hence \(x\otimes t\in\mathcal L\) for every
\(x\in\mathcal L\) and \(t\in\mathscr T(G)\), so
\(\mathcal L\) is a tensor ideal.

The Neeman--Thomason localization theorem  gives
\[
 \bigl(\mathscr T(G)/\mathcal L\bigr)^c
 \simeq
 \left(
   \frac{\mathscr K(G)}
        {\mathscr K_{\mathrm{ac}}(G)}
 \right)^{\natural}.
\]
By Proposition~\ref{prop5.2}, the
right-hand side is equivalent to \(D^b(kG)\). Realization induces a coproduct-preserving exact functor
\[
 \overline\Upsilon_G:
 \mathscr T(G)/\mathcal L
 \longrightarrow
 D(\mathcal A_G).
\]
Under the preceding identifications, its restriction to compact
objects is precisely the equivalence of
Proposition~\ref{prop5.2}. Hence
\(\overline\Upsilon_G\) is an equivalence on compact objects. Since
both categories are compactly generated, it is an equivalence. This concludes the proof.
\end{proof}

\begin{remark}\label{rem5.5}
Under the hypotheses of
Proposition~\ref{prop5.4}, we denote the fully faithful
coproduct-preserving right adjoint of \(\Upsilon_G\) by
\(\rho_G: K(\operatorname{Inj}\mathcal A_G)\longrightarrow\mathscr T(G).\)
\end{remark}

\begin{proposition}
\label{prop5.6}
Let \(G\) be a torsion-free compact \(p\)-adic Lie pro-\(p\) group of
dimension \(d\). Then
\(
 \operatorname{gldim}(\mathcal A_G)\leq d.
\)
\end{proposition}

\begin{proof}
Taking \(U=G\) in \cite[Proposition~2]{HeyerSchneider}, we obtain an
injective resolution of the tensor unit of length \(d\) whose terms are finite direct sums of \(C(G,k)\).
Let \(M\in\mathcal A_G\). Since \(k\) is a field, tensoring this
resolution over \(k\) with \(M\) preserves exactness. The isomorphisms
constructed in the proof of \cite[Lemma~5]{HeyerSchneider} give
natural \(G\)-equivariant isomorphisms
\[
 M\otimes_k C(G,k)
   \simeq M_{\mathrm{triv}}\otimes_k C(G,k)
   \simeq C(G,M_{\mathrm{triv}}).
\]
The last object is injective. Indeed, evaluation at the identity
gives a natural isomorphism of contravariant functors on
\(\mathcal A_G\):
\[
 \operatorname{Hom}_{\mathcal A_G}
 \bigl(-,C(G,M_{\mathrm{triv}})\bigr)
 \simeq
 \operatorname{Hom}_k
 \bigl(\operatorname{Res}^G_1(-),M\bigr).
\]
The restriction functor
\(
 \operatorname{Res}^G_1:\mathcal A_G\to\operatorname{Mod}(k)
\)
is exact, and \(\operatorname{Hom}_k(-,M)\) is exact because every
\(k\)-vector space is injective. Hence every object of
\(\mathcal A_G\) has injective dimension at most \(d\). This concludes the proof.
\end{proof}

\section{Selected restriction}
\label{sec:continuity}

In this section we introduce the stalkwise \(H\)-strata \(Z_G(H;x)\) and their associated localizing ideals
\(\mathcal L_G(H;x)\). We recover their compact parts from finite quotients and establish selected-restriction equivalences for finite
groups and countably based profinite groups.

\begin{notation}
For \(x\in X_G\), let \(\mathcal P_x\subseteq\mathscr K(G)\) be the
corresponding prime thick ideal and write
\[
 q_x:\mathscr T(G)\longrightarrow\mathscr T(G)_x
 :=\frac{\mathscr T(G)}
         {\operatorname{Loc}^{\otimes}(\mathcal P_x)}
\]
for the  localization.  We write \(j_x\) for the right adjoint of \(q_x\). Its compact spectrum is canonically
identified with \(
 \operatorname{Spc}\bigl(\mathscr T(G)_x^c\bigr)
 \simeq\operatorname{gen}_{X_G}(x).
\)
\end{notation}

\begin{definition}
\label{def6.2}
For \(x=\mathcal P_G(H,\mathfrak q)\), define the stalkwise \(H\)-stratum at \(x\) by
\[
 Z_G(H;x):=
 \sigma_G^{-1}([H]_G)\cap\operatorname{gen}_{X_G}(x).
\]
Whenever \(Z_G(H;x)\) is Thomason in the stalk spectrum, put
\begingroup
\setlength{\belowdisplayskip}{0pt}
\setlength{\belowdisplayshortskip}{0pt}
\[
 \begin{aligned}
 \bigl(\mathscr T(G)_x^c\bigr)_{Z_G(H;x)}
 &:={}
 \bigl\{a\in\mathscr T(G)_x^c
       \mid\operatorname{supp}(a)\subseteq Z_G(H;x)\bigr\},\\
 \mathcal L_G(H;x)
 &:={}
 \operatorname{Loc}^{\otimes}_{\mathscr T(G)_x}
 \langle\bigl(\mathscr T(G)_x^c\bigr)_{Z_G(H;x)}\rangle.
 \end{aligned}
\]
\endgroup
We denote the tensor unit of \(\mathcal L_G(H;x)\) by \(e_G(H;x)\).
\end{definition}

\begin{lemma}
\label{lem6.3}
Let \(G\) be profinite and let
\(x=\mathcal P_G(H,\mathfrak q)\), where \(H\) is topologically finitely generated.
There is an open normal subgroup \(N_0\trianglelefteq_oG\) such that
\[
 Z_G(H;x)=\mathrm{supp}_{X_G}(c_H)\cap \operatorname{gen}_{X_G}(x)
\]
where \(c_H:=
 \bigotimes_{M_i\in \mathrm{Max}_o(H)}
 \operatorname{kos}_G(M_iN_0)\in\mathscr K(G)\).  In particular,
\(Z_G(H;x)\) is a Thomason subset in \(
 \operatorname{Spc}\bigl(\mathscr T(G)_x^c\bigr)\).
\end{lemma}

\begin{proof}
Note that for a topologically finitely generated pro-\(p\) group \(H\), \(\mathrm{Max}_o(H)=\{M_1,\cdots, M_s\}\) is finite.
We first choose \(N_0\trianglelefteq_oG\) such that no conjugate of
\(H\) is contained in any \(M_iN_0\). For every \(i\), we claim that there exists
\(N_i\trianglelefteq_oG\) such that
\(
 H^g\nleq M_iN_i\)
for every \(g\in G.
\)
Fix \(i\) and suppose that no such \(N_i\) exists.
Then, for every \(N\trianglelefteq_oG\), the closed subset
\[
 T_{i,N}:=
 \{g\in G\mid H^g\leq M_iN\}
\]
would be nonempty. Moreover,
\(
 T_{i,N_1\cap\cdots\cap N_r}
 \subseteq
 T_{i,N_1}\cap\cdots\cap T_{i,N_r}.
\)
Thus the family \(\{T_{i,N}\}_N\) would have the finite-intersection
property. Compactness of \(G\) would give some
\(
 g\in\bigcap_{N\trianglelefteq_oG}T_{i,N}.
\)
Since \(M_i\) is closed, we have
\(
 H^g
 \leq
 \bigcap_{N\trianglelefteq_oG}M_iN
 =
 M_i
 <
 H,
\) which is impossible.

We may therefore choose, for every \(i\), an open normal subgroup
\(N_i\trianglelefteq_oG\) such that no conjugate of \(H\) is contained
in \(M_iN_i\). Put
\(
 N_0:=\bigcap_{i=1}^sN_i.
\)
Then \(N_0\trianglelefteq_oG\), and no conjugate of \(H\) is contained
in any \(M_iN_0\). Since \(N_0\) is normal and open, every \(M_iN_0\)
is an open subgroup of \(G\), so
\(c_H:=\bigotimes_{i=1}^s\operatorname{kos}_G(M_iN_0)\in\mathscr K(G)\) is defined and compact.

Let
\(
 y=\mathcal P_G(K,\mathfrak r)
 \in\operatorname{gen}_{X_G}(x).
\)
By \cite[Corollary~3.17]{BalmerGallauerArtin}, one has
\(K\leq_GH\). We also have
\[
 \begin{aligned}
 y\in\operatorname{supp}_{X_G}(c_H)
 &\quad\Longleftrightarrow\quad
 y\in
 \operatorname{supp}_{X_G}
 \bigl(\operatorname{kos}_G(M_iN_0)\bigr)
 \quad\text{for every }i\\
 &\quad\Longleftrightarrow\quad
 K\nleq_GM_iN_0
 \quad\text{for every }i.
 \end{aligned}
\]

Suppose first that the subgroup parameter of \(y\) is \([H]_G\).
Then \(K\) is \(G\)-conjugate to \(H\), and the choice of \(N_0\)
shows that
\(
 K\nleq_GM_iN_0\)
 for every \(i.
\)
Hence \(y\in\operatorname{supp}_{X_G}(c_H)\).

Conversely, suppose that the subgroup parameter of \(y\) is different
from \([H]_G\). Since \(K\leq_GH\), after conjugation we may assume
that \(K<H\). Every proper closed subgroup of a pro-\(p\) group is
contained in a maximal proper open subgroup, so
\(
 K\leq M_i
\)
for some \(i\). Therefore
\(
 K\leq_GM_iN_0,
\)
and thus
\(y\notin\operatorname{supp}_{X_G}(c_H)\). This concludes the proof.
\end{proof}

\begin{lemma}
\label{lem6.4}
Suppose that \(x=\mathcal P_G(H,\mathfrak q)\) is a weakly visible point in \(X_G\), and that \(Z_G(H;x)\) is the support of a compact object
\(c\in\mathscr T(G)_x^c\).  Then
\[
 \mathcal L_G(H;x)
 =\operatorname{Loc}^{\otimes}_{\mathscr T(G)_x}(c)
\]
is a compactly generated smashing ideal and is independent of
the choice of \(c\).  Moreover,
\(
 q_x(g_x)\simeq e_G(H;x)\otimes q_x(g_x),
\)
and \(q_x\) restricts to an equivalence of localizing ideals
\[
 q_x:
 \operatorname{Loc}^{\otimes}_{\mathscr T(G)}(g_x)
 \xrightarrow{\ \sim\ }
 \operatorname{Loc}^{\otimes}_{\mathscr T(G)_x}(q_x(g_x)).
\]
\end{lemma}
\begin{proof}
This is a routine exercise.
\end{proof}

\begin{corollary}
\label{cor6.5}
Under the hypotheses of
Lemma~\ref{lem6.4},
\(\mathscr T(G)\) satisfies minimality at
\(x=\mathcal P_G(H,\mathfrak q)\) if and only if
\(
 \operatorname{Loc}^{\otimes}_{\mathcal L_G(H;x)}
       \bigl(q_x(g_x)\bigr)
\)
is a minimal localizing ideal of \(\mathcal L_G(H;x)\).
\end{corollary}

\begin{proof}
Lemma~\ref{lem6.4} gives an equivalence
\[
 q_x:
 \operatorname{Loc}^{\otimes}_{\mathscr T(G)}(g_x)
 \xrightarrow{\sim}
 \operatorname{Loc}^{\otimes}_{\mathscr T(G)_x}
       \bigl(q_x(g_x)\bigr)
\]
and the identity \(
 q_x(g_x)\simeq e_G(H;x)\otimes q_x(g_x)
\) implies
\[
 \operatorname{Loc}^{\otimes}_{\mathscr T(G)_x}
       \bigl(q_x(g_x)\bigr)
 =
 \operatorname{Loc}^{\otimes}_{\mathcal L_G(H;x)}
       \bigl(q_x(g_x)\bigr).
\]
Suppose that
\(
 \mathcal J\subseteq
 \operatorname{Loc}^{\otimes}_{\mathscr T(G)_x}
       \bigl(q_x(g_x)\bigr)
\)
is a localizing ideal of \(\mathscr T(G)_x\).  The above
equality implies that \(\mathcal J\subseteq\mathcal L_G(H;x)\).
Moreover, for \(\ell\in\mathcal L_G(H;x)\) and \(j\in\mathcal J\), one
has \(\ell\otimes j\in\mathcal J\), since
\(\ell\in\mathscr T(G)_x\).  Hence \(\mathcal J\) is a localizing ideal of \(\mathcal L_G(H;x)\).

Conversely, suppose that
\(
 \mathcal J\subseteq
 \operatorname{Loc}^{\otimes}_{\mathcal L_G(H;x)}
       \bigl(q_x(g_x)\bigr)
\)
is a localizing ideal of \(\mathcal L_G(H;x)\).  For
\(t\in\mathscr T(G)_x\) and \(j\in\mathcal J\), we have
\(
 t\otimes j
 \simeq
 t\otimes e_G(H;x)\otimes j\in \mathcal J.
\)
Thus \(\mathcal J\) is a localizing  ideal of
\(\mathscr T(G)_x\). It follows that
\[
\begin{aligned}
&
 \operatorname{Loc}^{\otimes}_{\mathscr T(G)}(g_x)
 \text{ is minimal in }\mathscr T(G)
\\
&\qquad\Longleftrightarrow
 \operatorname{Loc}^{\otimes}_{\mathscr T(G)_x}
       \bigl(q_x(g_x)\bigr)
 \text{ is minimal in }\mathscr T(G)_x
\\
&\qquad\Longleftrightarrow
 \operatorname{Loc}^{\otimes}_{\mathcal L_G(H;x)}
       \bigl(q_x(g_x)\bigr)
 \text{ is minimal in }\mathcal L_G(H;x).
\end{aligned}\]
\end{proof}

\begin{lemma}
\label{lem6.6}
Let \(I\) be a small filtered poset. For \(i\leq j\), let
\(
 F_{ij}:\mathcal A_i\longrightarrow\mathcal A_j
\)
be exact strong symmetric monoidal functors between essentially small idempotent-complete
tt-categories with
\(
 F_{ii}=\operatorname{id}_{\mathcal A_i}\) and
 \(F_{jk}F_{ij}=F_{ik} \) for \(
 i\leq j\leq k.
\)
Let \(\mathcal A\) be an idempotent-complete essentially small
tt-category, and let
\(
 u_i:\mathcal A_i\longrightarrow\mathcal A
\)
be exact strong symmetric monoidal functors satisfying
\(
 u_jF_{ij}=u_i\)
 for \(i\leq j.\) Assume the following two conditions.
\begin{enumerate}
\item[\textup{(i)}]
For all \(i,j\in I\), \(a_i\in\mathcal A_i\), and
\(b_j\in\mathcal A_j\), the canonical map
\[
 \operatorname*{colim}_{k\geq i,j}
 \operatorname{Hom}_{\mathcal A_k}
 \bigl(F_{ik}(a_i),F_{jk}(b_j)\bigr)
 \xrightarrow{\ \sim\ }
 \operatorname{Hom}_{\mathcal A}
 \bigl(u_i(a_i),u_j(b_j)\bigr)
\]
is an isomorphism.

\item[\textup{(ii)}]
Every object of \(\mathcal A\) is a retract of \(u_i(a_i)\) for some
\(i\in I\) and some \(a_i\in\mathcal A_i\).
\end{enumerate}

Then for \(i_0\in I\) and \(c_{i_0}\in\mathcal A_{i_0}\), the canonical functor induces an exact tensor equivalence
\[
  \operatorname*{colim}_{i\geq i_0}
  \operatorname{thick}^{\otimes}_{\mathcal A_i}(F_{i_0i}(c_{i_0}))
 \xrightarrow{\ \sim\ }
 \operatorname{thick}^{\otimes}_{\mathcal A}(u_{i_0}(c_{i_0})).
\]

\end{lemma}

\begin{proof}
For \(i\geq i_0\), put
\(
 \mathcal J_i:=
 \operatorname{thick}^{\otimes}_{\mathcal A_i}(c_i),\)
 \(\mathcal J:=
 \operatorname{thick}^{\otimes}_{\mathcal A}(c)
\), where \(
 c_i:=F_{i_0i}(c_{i_0}),\)
\(
 c:=u_{i_0}(c_{i_0}).
\)
Since the transition functors are exact and strong symmetric
monoidal, they restrict to exact tensor functors
\[
 F_{ij}:\mathcal J_i\longrightarrow\mathcal J_j
 \qquad (i_0\leq i\leq j).
\]

Let \(\mathcal D\) be the filtered categorical colimit of the
\(\mathcal J_i\). Concretely, its objects may be represented by pairs
\((i,d_i)\), where \(i\geq i_0\) and \(d_i\in\mathcal J_i\), and
\[
 \operatorname{Hom}_{\mathcal D}
 \bigl((i,d_i),(j,e_j)\bigr)
 =
 \operatorname*{colim}_{k\geq i,j}
 \operatorname{Hom}_{\mathcal J_k}
 \bigl(F_{ik}(d_i),F_{jk}(e_j)\bigr).
\]
The cocone functors \(u_i\) induce a functor
\(
 v:\mathcal D\longrightarrow\mathcal A.
\)
Since every \(\mathcal J_k\) is a full subcategory of
\(\mathcal A_k\), condition \textup{(i)} shows that \(v\) is fully
faithful.

Let \(\mathcal E\) denote the essential image of \(v\). It is a
triangulated subcategory of \(\mathcal A\). Indeed, every morphism
between two objects in the image is represented, after passing to a
common stage, by a morphism
\(
 f_k:F_{ik}(d_i)\longrightarrow F_{jk}(e_j)
\)
in some \(\mathcal J_k\). Its cone belongs to \(\mathcal J_k\), and
exactness of \(u_k\) identifies the image of this cone with the cone
of the original morphism in \(\mathcal A\). Each \(\mathcal J_i\) is
idempotent complete, and filtered categorical colimits preserve
idempotent completeness. Hence \(\mathcal D\) is idempotent complete,
and the full faithfulness of \(v\) shows that \(\mathcal E\) is closed
under retracts. Thus \(\mathcal E\) is a thick triangulated subcategory
of \(\mathcal A\).

Since \(u_i(c_i)\simeq c\) and every \(u_i\) is exact and strong
symmetric monoidal, one has
\(
 u_i(\mathcal J_i)\subseteq\mathcal J.
\)
Hence
\(
 \mathcal E\subseteq\mathcal J.
\)
For the converse, first note that \(c\in\mathcal E\). We claim that
\(\mathcal E\) is a tensor ideal of \(\mathcal A\). Let
\(x\in\mathcal E\) and \(a\in\mathcal A\). By the definition of
\(\mathcal E\), one has \(x\simeq u_i(x_i)\) for some
\(x_i\in\mathcal J_i\). By condition \textup{(ii)}, the object
\(a\) is a retract of some \(u_j(a_j)\) with
\(a_j\in\mathcal A_j\). Choose \(k\geq i,j\). Then \(x\otimes a\)
is a retract of
\(
 u_k\bigl(F_{ik}(x_i)\otimes F_{jk}(a_j)\bigr).
\)
Since \(\mathcal J_k\) is a tensor ideal of \(\mathcal A_k\), one has
\(
 F_{ik}(x_i)\otimes F_{jk}(a_j)\in\mathcal J_k.
\)
Consequently, \(x\otimes a\in\mathcal E\), proving the claim.
Thus \(\mathcal E\) is a thick ideal of \(\mathcal A\)
containing \(c\). By the definition of \(\mathcal J\), this implies
\(
 \mathcal J\subseteq\mathcal E.
\)
Therefore \(\mathcal E=\mathcal J\), and \(v\) induces the desired exact tensor
equivalence.
\end{proof}

\begin{proposition}
\label{prop6.7}
Let \(G=\varprojlim_nG_n\) be an inverse limit of finite groups
with surjective transition maps. Let \(H\leq G\) be a topologically finitely generated closed pro-\(p\) group, and let
\(
 x=\mathcal P_G(H,\mathfrak q)=(x_n),
\)
and \(c_H\) be as in Lemma~\ref{lem6.3}.  Suppose that \(c_H\) is
defined over \(G_{n_0}\).  For \(n\geq n_0\), set
\(
 \mathcal L_n=
 \operatorname{Loc}^{\otimes}_{\mathscr T(G_n)_{x_n}}(c_n)
\), where \(c_n:=q_{x_n}\circ \mathrm{Infl}^{G_{n}}_{G_{n_{0}}}(c_H)\).
Then the inflation functors induce an exact tensor equivalence
\[
\operatorname*{colim}_{n\geq n_0}\mathcal L_n^c
 \xrightarrow{\ \sim\ }
 \mathcal L_G(H;x)^c.
\]
\end{proposition}

\begin{proof}
Let \(\mathcal P\) and \(\mathcal P_n\) be the prime thick ideals
corresponding to \(x\) and \(x_n\).  Identifying
\(\mathscr K(G_n)\) with its inflation image, one has \(
 \mathscr K(G)=\operatorname*{colim}_n\mathscr K(G_n)
 \) and \(
 \mathcal P_n=\mathscr K(G_n)\cap\mathcal P
\) by \cite[Propositions~2.4 and~3.1]{BalmerGallauerArtin}. The Neeman--Thomason localization theorem identifies the compact
objects in the stalks with
\(
 \bigl(\mathscr K(G_n)/\mathcal P_n\bigr)^\natural
\) and \(
 \bigl(\mathscr K(G)/\mathcal P\bigr)^\natural.
\) A morphism in the last Verdier quotient is represented by a roof
\[
 a\xleftarrow{s}d\xrightarrow{f}b,
 \qquad \operatorname{cone}(s)\in\mathcal P.
\]
The objects and morphisms occurring in this roof and in a chosen cone
triangle all occur in one \(\mathscr K(G_m)\) after increasing
\(m\), therefore the cone lies in \(\mathcal P_m\). This proves the surjectivity of
\[
 \operatorname*{colim}_{m\geq n}
 \operatorname{Hom}_{\mathscr K(G_m)/\mathcal P_m}(a_m,b_m)
 \longrightarrow
 \operatorname{Hom}_{\mathscr K(G)/\mathcal P}(a,b).
\]
For injectivity, suppose that two roofs represented in
\(
\mathscr K(G_n)/\mathcal P_n
\)
for some \(n\) have the same image in the final quotient.  By the calculus of
fractions, they admit a common refinement in
\(\mathscr K(G)\).  After passing to a sufficiently large \(m\), this
common refinement is defined in \(\mathscr K(G_m)\).  Since inflation
is fully faithful and
\(
 \mathcal P_m=\mathscr K(G_m)\cap\mathcal P
\),
it is already a common refinement in
\(\mathscr K(G_m)/\mathcal P_m\).  Hence
\[
 \operatorname*{colim}_{m\geq n}
 \operatorname{Hom}_{\mathscr K(G_m)/\mathcal P_m}(a_m,b_m)
 \xrightarrow{\ \sim\ }
 \operatorname{Hom}_{\mathscr K(G)/\mathcal P}(a,b).
\]

We next pass to idempotent completions. An object of
\(
 \bigl(\mathscr K(G)/\mathcal P\bigr)^\natural
\)
is represented by a pair \((a,e)\), where \(a\) belongs to
\(\mathscr K(G)/\mathcal P\) and
\(
 e:a\longrightarrow a\) satisfies
\(
 e^2=e.
\)
Choose \(m\) such that \(a\) is represented over \(G_m\). By the
surjectivity, there exist \(r\geq m\) and
\(
 e_r:a_r\longrightarrow a_r
\)
whose image is \(e\). Since \(e^2=e\), the injectivity gives \(s\geq r\) such that
\(
 e_s^2=e_s
\)
in \(\mathscr K(G_s)/\mathcal P_s\). Thus \((a,e)\) lies in the
essential image of
\[
 \bigl(\mathscr K(G_s)/\mathcal P_s\bigr)^\natural
 \longrightarrow
 \bigl(\mathscr K(G)/\mathcal P\bigr)^\natural.
\]
Moreover, for two retracts \((a,e)\) and \((b,f)\), their morphism
group is
\(
 f\,
 \operatorname{Hom}_{\mathscr K(G)/\mathcal P}(a,b)\,
 e.
\)
Consequently, the required Hom isomorphism holds
after idempotent completion.

The inflation functors are exact tensor functors and descend to the
compatible stalks. Since \(c_n\) is
compact,
\(
 \mathcal L_n^c
 =
 \operatorname{thick}^{\otimes}_{\mathscr T(G_n)_{x_n}^c}(c_n).
\)
On the other hand, Lemma~\ref{lem6.3} shows
\[
 \mathcal L_G(H;x)^c
 =
 \operatorname{thick}^{\otimes}_{\mathscr T(G)_x^c}
 \bigl(q_x(c_H)\bigr).
\]
Then the conclusion follows from
Lemma~\ref{lem6.6}.
\end{proof}

\begin{remark}\label{rem6.8}
Let the notation be as in  Proposition~\ref{prop6.7}. If \(G_n=G/U_n\), where \(U_n\leq N_0\) and
\(U_n\cap H\leq H^p\), and if \(H_n=HU_n/U_n\), then
\(
 \operatorname{supp}_{\mathscr T(G_n)_{x_n}}(c_n)
 =
 Z_{G_n}(H_n;x_n).
\)
Consequently,
\(
 \mathcal L_n=\mathcal L_{G_n}(H_n;x_n).
\)
\end{remark}

\begin{remark}
Proposition~\ref{prop6.7} asserts a filtered
\(2\)-colimit only for the compact triangulated categories.  It makes
no assertion that the big categories \(\mathcal L_n\) have
\(\mathcal L_G(H;x)\) as their filtered colimit.
\end{remark}

For a finite group \(G\), we write \(\operatorname{Max}(G)\) for the set of maximal proper subgroups of \(G\).

\begin{lemma}
\label{lem6.10}
Let \(Q\) be finite and \(H\leq Q\)  a \(p\)-subgroup, and let
\(x=\mathcal P_Q(H,\mathfrak q)\).  Then we have \[\mathrm{supp}(c_{Q,H})\cap \mathrm{gen}_{X_{Q}}(x)=Z_Q(H;x)\]
where \(
 c_{Q,H}:=
 \bigotimes_{M\in\operatorname{Max}(H)}\operatorname{kos}_Q(M).
\)
\end{lemma}

\begin{proof}
Let \(\xi=\mathcal P_Q(K,\mathfrak r)\) be a generalization of \(x\).
Then \(K\leq_QH\), so after conjugation we may assume \(K\leq H\).
The Koszul support formula gives
\[
 \xi\in\operatorname{supp}(c_{Q,H})
 \quad\Longleftrightarrow\quad
 K\nleq_QM
 \quad\text{for every }M\in\operatorname{Max}(H).
\]
Every proper subgroup of the finite \(p\)-group \(H\) lies in a
maximal proper subgroup, while \(H\) cannot be subconjugate to any such
subgroup.  Hence \(K\nleq_Q M\) for every
\(M\in\operatorname{Max}(H)\) if and only if
\([K]_Q=[H]_Q\).
\end{proof}

Let \(Q,H,x\) be as in
Lemma~\ref{lem6.10}, and let
\(N_Q(H)\leq L\leq Q\). Since
\(
 N_L(H)=N_Q(H),
\)
the corresponding Weyl groups agree:
\(
 N_L(H)/H=N_Q(H)/H.
\)
Thus we may regard
\(
 \mathfrak q\in\mathcal V_{N_Q(H)/H}
\)
as a point of
\(
 \mathcal V_{N_L(H)/H}
\),
and write
\[
\begin{aligned}
 x_L&:=\mathcal P_L(H,\mathfrak q),\\
 e_Q&:=e_Q(H;x),\\
 e_L&:=e_L(H;x_L),\\
 \rho_L&:=
 \operatorname{Spc}(\operatorname{Res}_L^Q):
 X_L\longrightarrow X_Q.
\end{aligned}
\]

\begin{proposition}
\label{prop6.11}
With the preceding notation, the restriction induces a coproduct-preserving exact tensor
functor
\(
 \overline{\operatorname{Res}}_L^Q:
 \mathscr T(Q)_x
 \longrightarrow
 \mathscr T(L)_{x_L}
\)
such that
\(
 \overline{\operatorname{Res}}_L^Q(e_Q)\simeq e_L.
\)
Moreover, \(
 \overline{\operatorname{Res}}_L^Q\big|_{\mathcal L_Q(H;x)}:\mathcal L_Q(H;x) \rightarrow
\mathcal L_L(H;x_L)
\) has  an exact coproduct-preserving right adjoint
\[
 I_{L,Q}^{H,x}
 :\mathcal L_L(H;x_L)\rightarrow \mathcal L_Q(H;x)
\]
such that
\[
 I_{L,Q}^{H,x}(Y)\simeq
 e_Q\otimes q_x
 \operatorname{Coind}_L^Q\bigl(j_{x_L}Y\bigr)
 \simeq
 e_Q\otimes q_x
 \operatorname{Ind}_L^Q\bigl(j_{x_L}Y\bigr).
\]
\end{proposition}

\begin{proof}
We first claim that \(
 \rho_L^{-1}\bigl(Z_Q(H;x)\bigr)
 \cap\operatorname{gen}_{X_L}(x_L)
 =
 Z_L(H;x_L).
\)
By \cite[Remark~7.6(b)]{BalmerGallauerGeometry} and
\(N_L(H)/H=N_Q(H)/H\), we have the equality  \(\rho_L(x_L)=x\). Let
\(
 z=\mathcal P_L(K,\mathfrak r)
 \in\operatorname{gen}_{X_L}(x_L).
\)
By \cite[Corollary~3.17]{BalmerGallauerArtin}, one has
\(K\leq_LH\). If \(\rho_L(z)\in Z_Q(H;x)\), then
\([K]_Q=[H]_Q\). After replacing \(K\) by an \(L\)-conjugate, assume
\(K\leq H\). Since \(K\) and \(H\) are \(Q\)-conjugate, they have the
same order, and hence \(K=H\). Thus \(z\in Z_L(H;x_L)\). Conversely, let \(z\in Z_L(H;x_L)\). After \(L\)-conjugating, we can write
\(
 z=\mathcal P_L(H,\mathfrak r).
\)
Then by
\cite[Remark~7.6(b)]{BalmerGallauerGeometry}
we have
\(
 \rho_L(z)
 =
 \mathcal P_Q\bigl(H,\bar\rho_{L,H}(\mathfrak r)\bigr).
\)
Hence \(\rho_L(z)\) has subgroup parameter \([H]_Q\). Since \(\rho_L\) is continuous and \(\rho_L(x_L)=x\), it follows that
\(
 x\in\overline{\{\rho_L(z)\}}.
\)
Therefore we have
\(
 \rho_L(z)\in Z_Q(H;x)
\), which ends the claim.

Put \(F=\operatorname{Res}_L^Q\). The equality
\(\rho_L(x_L)=x\) says
\(
 (F^c)^{-1}(\mathcal P_{x_L})=\mathcal P_x.
\)
It follows that
\(
 F\bigl(\operatorname{Loc}^{\otimes}(\mathcal P_x)\bigr)
 \subseteq
 \operatorname{Loc}^{\otimes}(\mathcal P_{x_L}),
\)
so the Verdier-quotient universal property gives
\[
 \overline{\operatorname{Res}}_L^Q:
 \mathscr T(Q)_x\longrightarrow\mathscr T(L)_{x_L}.
\]
Then \(
 \overline{\operatorname{Res}}_L^Q(e_Q)\simeq e_L
\) follows from \cite[Theorem~6.3]{BalmerFavi} and the claim above, and so the restriction maps \(\mathcal L_Q(H;x)\) into
\(\mathcal L_L(H;x_L)\).

Let \(U=\operatorname{Coind}_L^Q\), the right adjoint of \(F\). For
\(A\in\ker(q_x)\) and \(Y\in\mathscr T(L)_{x_L}\), one has
\[
 \begin{aligned}
 \operatorname{Hom}_{\mathscr T(Q)}
 \bigl(A,U(j_{x_L}Y)\bigr)
 &\simeq
 \operatorname{Hom}_{\mathscr T(L)}
 \bigl(F(A),j_{x_L}Y\bigr)\\
 &\simeq
 \operatorname{Hom}_{\mathscr T(L)_{x_L}}
 \bigl(q_{x_L}F(A),Y\bigr)
 =0.
 \end{aligned}
\]
Hence \(U(j_{x_L}Y)\) is \(q_x\)-local, and
\(j_xq_xU(j_{x_L}Y)\simeq U(j_{x_L}Y)\). For \(T\in\mathscr T(Q)_x\) and
\(Y\in\mathscr T(L)_{x_L}\), there are natural isomorphisms
\[
\begin{aligned}
\operatorname{Hom}_{\mathscr T(L)_{x_L}}
 \bigl(\overline{\operatorname{Res}}_L^Q(T),Y\bigr)
&\simeq
 \operatorname{Hom}_{\mathscr T(L)_{x_L}}
 \bigl(q_{x_L}F(j_xT),Y\bigr)\\
&\simeq
 \operatorname{Hom}_{\mathscr T(L)}
 \bigl(F(j_xT),j_{x_L}Y\bigr)\\
&\simeq
 \operatorname{Hom}_{\mathscr T(Q)}
 \bigl(j_xT,U(j_{x_L}Y)\bigr)\\
&\simeq
 \operatorname{Hom}_{\mathscr T(Q)}
 \bigl(j_xT,j_xq_xU(j_{x_L}Y)\bigr)\\
&\simeq
 \operatorname{Hom}_{\mathscr T(Q)_x}
 \bigl(T,q_xU(j_{x_L}Y)\bigr).
\end{aligned}
\]
 Hence
\(
 \overline{\operatorname{Res}}_L^Q
 \dashv
 q_xUj_{x_L}.
\)
Finally, \(e_Q\otimes-\) is right adjoint to the inclusion of
\(\mathcal L_Q(H;x)\) into \(\mathscr T(Q)_x\). This proves
\[
 \overline{\operatorname{Res}}_L^Q\big|_{\mathcal L_Q(H;x)}
 \dashv I_{L,Q}^{H,x}:=e_Q\otimes q_xUj_{x_L}.
\]
The localization \(q_{x_L}\) is generated by compact objects, so
\(j_{x_L}\) preserves coproducts. Since \(Q/L\) is finite,
coinduction is exact and preserves coproducts. Hence
\(I_{L,Q}^{H,x}\) has the same properties.
\end{proof}

\begin{definition}
\label{def6.12}
The functor
\[
 \begin{aligned}
 R_{Q,L}^{H,x}:
 \mathcal L_Q(H;x)&\longrightarrow\mathcal L_L(H;x_L),
 \\
 X&\longmapsto
 e_L\otimes\overline{\operatorname{Res}}_L^Q(X)
 \end{aligned}
\]
is called \emph{selected restriction}. Its right adjoint
\(I_{L,Q}^{H,x}\),  is called
\emph{selected coinduction}. When \(L=N_Q(H)\), we abbreviate these
functors to \(R\) and \(I\).
\end{definition}

\begin{proposition}
\label{prop6.13}
Let \(Q,H,x,L\) be as in Definition
\ref{def6.12}. Then
\[
 R_{Q,L}^{H,x}:
 \mathcal L_Q(H;x)
 \xrightarrow{\ \sim\ }
 \mathcal L_L(H;x_L)
\]
is a tensor-triangulated equivalence, with quasi-inverse
\(I_{L,Q}^{H,x}\).
\end{proposition}

\begin{proof}
We first take \(L=N=N_Q(H)\). By
Proposition~\ref{prop6.11}, there is an
adjunction
\(
 R\dashv I.
\)
Put
\(
 \widetilde e_N:=j_{x_N}(e_N)\in\mathscr T(N).
\)
The fully faithful embedding of the smashing-local subcategory is
monoidal on local objects. Hence, for
\(Y\in\mathcal L_N(H;x_N)\), we have
\(
 j_{x_N}(Y)
 \simeq
 \widetilde e_N\otimes j_{x_N}(Y).
\)
It is routine to check that the Balmer-Favi support
of \(\widetilde e_N\) in \(X_N\) is exactly the subset
\(Z_N(H;x_N)\subseteq\operatorname{gen}_{X_N}(x_N)\).

Using \(R(e_Q)\simeq e_N\), the composite of the adjoint functors is
given by
\[
 RI(Y)\simeq
 e_N\otimes q_{x_N}\!\left(
   \operatorname{Res}_N^Q
   \operatorname{Coind}_N^Q(j_{x_N}Y)
 \right).
\]

For \(g\in Q\), put \(K_g=N\cap{}^gN\). Using
\(\operatorname{Coind}_N^Q\simeq\operatorname{Ind}_N^Q\), the Mackey
formula expresses
\[
 \operatorname{Res}_N^Q
 \operatorname{Coind}_N^Q\bigl(j_{x_N}Y\bigr)
\]
as a finite direct sum indexed by \(N\backslash Q/N\), whose
\(g\)-summand is
\(
 \operatorname{Ind}_{K_g}^N
 \operatorname{Res}_{K_g}^{{}^gN}
 \bigl({}^gj_{x_N}Y\bigr).
\)
The identity-double-coset summand is \(j_{x_N}Y\).

Suppose that \(g\notin N\). Since \(
 j_{x_N}(Y)
 \simeq
 \widetilde e_N\otimes j_{x_N}(Y)
\), we have
\[\operatorname{Ind}_{K_g}^N
 \operatorname{Res}_{K_g}^{{}^gN}
 \bigl({}^gj_{x_N}Y\bigr)\otimes \widetilde e_N\simeq
 \operatorname{Ind}_{K_g}^N\!\left(
  \operatorname{Res}_{K_g}^N(\widetilde e_N)
  \otimes
  \operatorname{Res}_{K_g}^{{}^gN}({}^g\widetilde e_N)
  \otimes
  \operatorname{Res}_{K_g}^{{}^gN}({}^gj_{x_N}Y)
 \right).
\]
Let
\(
 \rho_1:X_{K_g}\longrightarrow X_N\)
and \(
 \rho_2:X_{K_g}\longrightarrow X_{{}^gN}
\)
be the spectral maps induced by
\(\operatorname{Res}_{K_g}^{N}\) and
\(\operatorname{Res}_{K_g}^{{}^gN}\), respectively.
For the Balmer-Favi support, we have
\[
 \begin{aligned}
 \operatorname{Supp}_{X_{K_g}}
 \bigl(\operatorname{Res}_{K_g}^{N}(\widetilde e_N)\bigr)
 &=
 \rho_1^{-1}\bigl(Z_N(H;x_N)\bigr),\\
 \operatorname{Supp}_{X_{K_g}}
 \bigl(\operatorname{Res}_{K_g}^{{}^gN}
 ({}^g\widetilde e_N)\bigr)
 &=
 \rho_2^{-1}\bigl({}^gZ_N(H;x_N)\bigr).
 \end{aligned}
\]
 These inverse images are disjoint. Indeed, a point
in both supports would have a subgroup parameter represented by a
subgroup \(J\leq K_g\) that is \(N\)-conjugate to \(H\) and
\({}^gN\)-conjugate to \({}^gH\). Since \(N=N_Q(H)\), every
\(N\)-conjugate of \(H\) is \(H\), and every \({}^gN\)-conjugate of
\({}^gH\) is \({}^gH\). Thus \(H=J={}^gH\), which would imply
\(g\in N_Q(H)=N\), a contradiction. The finite-group categories are
stratified by \cite[Theorem~9.11]{BalmerGallauerGeometry}, so empty
support detects the zero object. Therefore \(\operatorname{Res}_{K_g}^N(\widetilde e_N)
  \otimes
  \operatorname{Res}_{K_g}^{{}^gN}({}^g\widetilde e_N)\) is zero, and so every nonidentity double-coset summand is annihilated by
\(\widetilde e_N\). It follows that the counit of \(R\dashv I\),
\(
 \epsilon:RI\xrightarrow{\ \sim\ }\operatorname{id},
\)
is an isomorphism.

The functor \(R\) is conservative. Indeed, if
\(0\neq X\in\mathcal L_Q(H;x)\), finite-group stratification supplies
a point
\[
 \xi=\mathcal P_Q(H,\mathfrak r)\in Z_Q(H;x)
\]
such that \(g_\xi\otimes X\neq0\), and minimality of the localizing ideal
\(
 \Gamma_\xi\bigl(\mathscr T(Q)_x\bigr)=
 \operatorname{Loc}^{\otimes}_{\mathscr T(Q)_x}(g_\xi)
\) implies
\(g_\xi\in\operatorname{Loc}^{\otimes}(X)\). Put
\(
 \xi_N=\mathcal P_N(H,\mathfrak r).
\)
Since \(N=N_Q(H)\), the Weyl groups of \(H\) in \(N\) and \(Q\)
are both \(N/H\). Then
\(
 \rho_N^{-1}(\{\xi\})\cap Z_N(H;x_N)=\{\xi_N\}.
\)
Therefore, we have
\[
 \begin{aligned}
 R(g_\xi)
 &=
 e_N\otimes\overline{\operatorname{Res}}_N^Q(g_\xi)\\
 &\simeq
 e_N\otimes g_{\rho_N^{-1}(\{\xi\})}\\
 &\simeq
 g_{\rho_N^{-1}(\{\xi\})\cap Z_N(H;x_N)}\\
 &\simeq g_{\xi_N}\\
 &\neq 0.
 \end{aligned}
\]
If \(R(X)=0\), applying \(R\) to
\(g_\xi\in\operatorname{Loc}^{\otimes}(X)\) would give
\(R(g_\xi)=0\), a contradiction. Thus \(R\) is conservative.

Let
\(
 \eta:\operatorname{id}\longrightarrow IR
\)
be the unit. Since \(\epsilon\) is an isomorphism, the triangle
identity shows that \(R(\eta)\) is an isomorphism. Conservativity of
\(R\) implies that \(\eta\) is an isomorphism. Thus \(R\) and \(I\)
are quasi-inverse equivalences.

For a general \(L\) containing \(N_Q(H)\), set \(N=N_Q(H)\). Since \(N\leq L\), one has
\(N_L(H)=N\). Transitivity of restriction, together with
idempotent base change, gives a natural isomorphism
\[
 R_{L,N}^{H,x_L}\circ R_{Q,L}^{H,x}
 \xrightarrow{\ \sim\ }
 R_{Q,N}^{H,x}.
\]
The normalizer case, applied to the ambient groups \(Q\) and \(L\),
shows that both \(R_{Q,N}^{H,x}\) and
\(R_{L,N}^{H,x_L}\) are equivalences. Hence
\(
 R_{Q,L}^{H,x}
 \simeq
 \bigl(R_{L,N}^{H,x_L}\bigr)^{-1}
 \circ R_{Q,N}^{H,x}
\)
is an equivalence.
\end{proof}

\begin{lemma}\label{lem6.14}
Let \(G\) be profinite and let
\(x=\mathcal P_G(H,\mathfrak q)\), where \(H\) is topologically finitely generated. Suppose \(L\leq_oG\) contains
\(N_G(H)\), and  put
\(
 x_L:=\mathcal P_L(H,\mathfrak q).
\)
Then the restriction induces a coproduct-preserving exact tensor
functor
\(
 \overline{\operatorname{Res}}_L^G:
 \mathscr T(G)_x
 \longrightarrow
 \mathscr T(L)_{x_L}
\)
such that
\(
 \overline{\operatorname{Res}}_L^G(e_G(H;x))\simeq e_L(H;x_L).
\)
Moreover, \(
 \overline{\operatorname{Res}}_L^G\big|_{\mathcal L_G(H;x)}:\mathcal L_G(H;x) \rightarrow
\mathcal L_L(H;x_L)
\) has  an exact coproduct-preserving right adjoint
\[
 \iota_{L,G}^{H,x}
 :\mathcal L_L(H;x_L)\rightarrow \mathcal L_G(H;x)
\]
such that
\[
 \iota_{L,G}^{H,x}(Y)\simeq
 e_G(H;x)\otimes q_x
 \operatorname{Coind}_L^G\bigl(j_{x_L}Y\bigr).
\]
\end{lemma}

\begin{proof}
 Note that \(Z_G(H;x)\) is a Thomason subset by Lemma~\ref{lem6.3}. The proof is similar to that of Proposition~\ref{prop6.11}.
\end{proof}

\begin{definition}
\label{def6.15}
Under the hypotheses of
Lemma~\ref{lem6.14}, define
\emph{selected restriction} by
\[
 \rho_{G,L}^{H,x}(X):=
 e_L(H;x_L)\otimes\overline{\operatorname{Res}}_L^G(X).
\]
We refer to its right adjoint \(\iota_{L,G}^{H,x}\) as the \emph{selected coinduction}.
\end{definition}

\begin{theorem}
\label{thm6.16}
Let \(G\) be a countably based profinite group  and let
\(x=\mathcal P_G(H,\mathfrak q)\), where \(H\) is topologically finitely generated. Suppose \(L\leq_oG\) contains
\(N_G(H)\), and put
\(
 x_L:=\mathcal P_L(H,\mathfrak q).
\)  Then selected
restriction induces a tensor-triangulated equivalence
\[
 \rho_{G,L}^{H,x}:
 \mathcal L_G(H;x)\xrightarrow{\ \sim\ }
 \mathcal L_L(H;x_L)
\]
with quasi-inverse \(\iota_{L,G}^{H,x}\).
\end{theorem}

\begin{proof}
Let \(c_H^G\) be the compact object supplied by
Lemma~\ref{lem6.3}, and put
\(
 c_H^L:=\operatorname{Res}_L^G(c_H^G).
\)
Then we have
\[
 \begin{aligned}
 \operatorname{supp}_{\mathscr T(L)_{x_L}}
   \bigl(q_{x_L}(c_H^L)\bigr)
 &=
 \rho_L^{-1}\bigl(
   \operatorname{supp}_{X_G}(c_H^G)\bigr)
 \cap\operatorname{gen}_{X_L}(x_L)\\
 &=
 \rho_L^{-1}\bigl(Z_G(H;x)\bigr)
 \cap\operatorname{gen}_{X_L}(x_L)\\
 &=
 Z_L(H;x_L).
 \end{aligned}
\]

Since \(H\) is topologically finitely generated,
\(\Phi(H)\) is open in \(H\). Hence there exists
\(V\trianglelefteq_oG\) such that
\(
 V\cap H\leq\Phi(H).
\)
Choose a cofinal descending sequence
\[
 U_n\trianglelefteq_oG
\]
such that
\(
 U_n\leq L\cap N_0\cap V.
\)
In particular,
\(
 U_n\cap H\leq\Phi(H).
\)
Since \(c_H^G\) is inflated from \(G/N_0\) and \(U_n\leq N_0\), it is
also inflated from \(G/U_n\). Put
\(A_n:=H\cap U_n,\)\(
 G_n:=G/U_n,\)\(
 L_n:=L/U_n,\)\(
 H_n:=HU_n/U_n\simeq H/A_n.\) Then by the compatibility of
restriction with inflation, we have \[
 \operatorname{Infl}_{G_n}^{G}(c_{H,n}^G)\simeq c_H^G,
 \qquad
 \operatorname{Infl}_{L_n}^{L}(c_{H,n}^L)\simeq c_H^L.
\] where \(
 c_{H,n}^G:=
 \bigotimes_{M_i\in \mathrm{Max}_o(H)}
 \operatorname{kos}_{G_n}(M_iN_0/U_n),\) \(
 c_{H,n}^L:=
 \operatorname{Res}_{L_n}^{G_n}(c_{H,n}^G).
\)
Let \(x_n\) and \(x_{L,n}\) denote the images of \(x\) and \(x_L\) in
the corresponding finite-group spectra. By definition we have \(A_n\leq M_i\) for every \(i\), and the subgroups \(M_i/A_n\) are precisely the maximal
subgroups of \(H_n\).  By \cite[Lemma~3.16]{BalmerGallauerArtin}, together
with the defining property of \(N_0\), we have
\[
 \operatorname{supp}_{X_{G_n}}(c_{H,n}^G)
 \cap\operatorname{gen}_{X_{G_n}}(x_n)
 =
 Z_{G_n}(H_n;x_n)
\]
and similarly
\[
 \operatorname{supp}_{X_{L_n}}(c_{H,n}^L)
 \cap\operatorname{gen}_{X_{L_n}}(x_{L,n})
=Z_{L_n}(H_n;x_{L,n}).
\]
For all sufficiently large \(n\),
\(
 N_{G_n}(H_n)\leq L_n.
\)
Suppose otherwise. After passing to a cofinal subsequence, choose
\(g_n\in G\setminus L\) whose image in \(G_n\) normalizes \(H_n\).
Since \(G\setminus L\) is compact, after passing to a subsequence we
may assume that \(g_n\to g\in G\setminus L\). For every fixed \(m\),
the image of \(g_n\) normalizes \(H_m\) for all sufficiently large
\(n\), and hence
\(
 H^{g}U_m=HU_m.
\)
Intersecting over \(m\) gives \(H^g=H\), contradicting
\(N_G(H)\leq L\).

After discarding finitely many terms, assume that
\(
 N_{G_n}(H_n)\leq L_n\)
for every \(n\).
Let
\(
 q_n:G_{n+1}\twoheadrightarrow G_n,\)
and \(
 \ell_n:L_{n+1}\twoheadrightarrow L_n \)
be the  quotient maps, and put
\(\mathcal E_n:=
 \mathcal L_{G_n}(H_n;x_n)^c,\) \(\mathcal F_n:=
 \mathcal L_{L_n}(H_n;x_{L,n})^c.\) By Lemma~\ref{lem6.4}, we have
\(
 \mathcal E_n=
 \operatorname{thick}^{\otimes}
 \bigl(q_{x_n}(c_{H,n}^G)\bigr)\) and \(
 \mathcal F_n=
 \operatorname{thick}^{\otimes}
 \bigl(q_{x_{L,n}}(c_{H,n}^L)\bigr).
\) On the other hand, we have
\(
 q_n^*(c_{H,n}^G)\simeq c_{H,n+1}^G\) and
 \(\ell_n^*(c_{H,n}^L)\simeq c_{H,n+1}^L.
\)
Hence inflation restricts to tensor exact functors
\(
 q_n^*:\mathcal E_n\longrightarrow\mathcal E_{n+1}\) and
\( \ell_n^*:\mathcal F_n\longrightarrow\mathcal F_{n+1}.
\)

Let \(S_n:=\bigl(R_{G_n,L_n}^{H_n,x_n}\bigr)^c:
 \mathcal E_n\longrightarrow\mathcal F_n\)
be the selected restriction on compact objects. Since
\(N_{G_n}(H_n)\leq L_n\),
Proposition~\ref{prop6.13} shows that \(S_n\) is
an equivalence. By Proposition~\ref{prop6.11}, there are canonical natural isomorphisms
\(\theta_n:
 S_{n+1}q_n^*
 \xrightarrow{\ \sim\ }
 \ell_n^*S_n\) which are compatible with iterated quotient maps.

Let \(
 u_n^G:\mathcal E_n\longrightarrow\mathcal E\) and
\( u_n^L:\mathcal F_n\longrightarrow\mathcal F
\)
be induced  inflations where \(
 \mathcal E:=\mathcal L_G(H;x)^c,\) \(
 \mathcal F:=\mathcal L_L(H;x_L)^c\). The Hom-colimit and idempotent-descent arguments in the proof of
Proposition~\ref{prop6.7} verify conditions of Lemma~\ref{lem6.6} for both towers. Applying that lemma to \(q_{x_{n}}(c_{H,n}^G)\) and \(q_{x_{L,n}}(c_{H,n}^L)\) gives
\[
  \operatorname*{colim}_n\mathcal E_n
 \xrightarrow{\ \sim\ }
 \mathcal E,
 \qquad
  \operatorname*{colim}_n\mathcal F_n
 \xrightarrow{\ \sim\ }
 \mathcal F.
\]
Hence
\[
 (\rho_{G,L}^{H,x})^c:
 \mathcal L_G(H;x)^c
 \xrightarrow{\ \sim\ }
 \mathcal L_L(H;x_L)^c
\]
is an equivalence, and thus we obtain the desired tensor-triangulated equivalence.
\end{proof}

\section{Pointwise minimality and stratification}\label{sec:minimality}

In this section we prove pointwise minimality for \(\mathscr T(P)\) by treating separately the four types of closed subgroups  described in
Lemma~\ref{lem4.3}. Together with the local-to-global principle established in
Theorem~\ref{thm4.7}, this yields the stratification of \(\mathscr T(P)\) for every \(p\)-decomposition group \(P\).

\begin{proposition}\label{prop7.1}
Let \(P\) be a \(p\)-decomposition group and \(x=\mathcal P_P(1,\mathfrak q)\in X_P\). Then \(\mathscr T(P)\) satisfies minimality at \(x\).
\end{proposition}
\begin{proof}
Since \(P\) is a
torsion-free compact \(p\)-adic Lie pro-\(p\) group of dimension \(2\), by
Proposition~\ref{prop5.4} and Proposition~\ref{prop5.6},  there is
a  finite localization
\[\frac{\mathscr T(P)}
      {\operatorname{Loc}^{\otimes}
       (\mathscr K_{\mathrm{ac}}(P))}
 \simeq
 D\bigl(\mathcal A_P)
 \simeq
 K\bigl(\operatorname{Inj}
        \mathcal A_P)
\]
where \(D\bigl(\mathcal A_P)\) has only trivial localizing ideals by Theorem~\ref{thm5.1}.
By Proposition~\ref{prop4.4}, the spectrum
\(\mathcal V_P\) is the singleton \(\{x\}\).  Thus \(K\bigl(\operatorname{Inj}
        \mathcal A_P)\) satisfies minimality at \(x\). Since \(\mathscr T(P)\) satisfies the local-to-global principle by Theorem~\ref{thm4.7} and  \(X_P\) is weakly noetherian by Proposition~\ref{prop4.5}, the conclusion follows from \cite[Remark~5.4]{BarthelHeardSanders}.
\end{proof}

\begin{lemma}\label{lem7.2}
Let \(F:\mathcal T\to\mathcal S\) be a geometric functor between
rigidly-compactly generated tt-categories with weakly noetherian
spectra, and let \(U\) be its right adjoint. Suppose that:
\begin{enumerate}[label=\textup{(\roman*)},leftmargin=2.0em]
\item
\(\operatorname{Spc}(F)^{-1}(x)=\{y\}\) for some \(x\in \operatorname{Spc}(\mathcal T^{c})\) and \(y\in \operatorname{Spc}(\mathcal S^{c})\);

\item
\(\Gamma_y\mathcal S\) is a minimal localizing ideal;

\item
\(
 g_x\in
 \operatorname{Loc}_{\mathcal T}^{\otimes}
 \bigl(U(\mathbf{1}_{\mathcal S})\otimes g_x\bigr).
\)
\end{enumerate}
Then \(\Gamma_x\mathcal T\) is minimal.
\end{lemma}
\begin{proof}
The point-idempotent base-change formula
\cite[Theorem~6.3]{BalmerFavi} gives \(F(g_x)\simeq g_y\). Let
\(0\neq t\in\operatorname{Loc}^{\otimes}(g_x)\). If \(F(t)=0\), the
projection formula gives \(U(\mathbf{1}_{\mathcal S})\otimes t=0\). Then  by~\textup{(iii)}, we have \(t\simeq g_x\otimes t=0\), a contradiction. Thus
\(F(t)\neq0\). Since \(F(t)\simeq F(g_x)\otimes F(t)\simeq g_y\otimes F(t)\in \Gamma_y\mathcal S\), then we have \(g_y\in\operatorname{Loc}^{\otimes}(F(t))\) by \textup{(ii)}. Applying the
coproduct-preserving functor \(U\) and using the projection formula
shows
\[
 U(\mathbf{1}_{\mathcal S})\otimes g_x\simeq U(F(g_x))\simeq U(g_y)\in\operatorname{Loc}^{\otimes}(t).
\]
Then the conclusion follows immediately from condition \textup{(iii)}.
\end{proof}

\begin{lemma}
\label{lem7.3}
Let \(P\) be a profinite pro-\(p\) group, \(V\trianglelefteq_oP\) and let \(R\) be the right adjoint of the modular \(V\)-fixed-points functor \(\Psi^V:\mathscr T(P)\longrightarrow\mathscr T(P/V).\)  Then \(R(1)\) is represented by a nonnegative
complex \[
 s=(\cdots\longrightarrow s_2\longrightarrow s_1
       \longrightarrow s_0\longrightarrow0)
\]
whose terms are set-indexed coproducts of finite permutation modules and
such that

\[
 s_0=k,\qquad
 s_1=\bigoplus_{\substack{K\leq_oP\\V\nleq K}}k(P/K).
\]

\end{lemma}\begin{proof}
The proof is the same as in \cite[Lemma~9.2]{BalmerGallauerGeometry}.
\end{proof}

\begin{lemma}\label{lem7.4}
Let \(G\) be profinite and let
\[
 s=(\cdots\longrightarrow s_2\longrightarrow s_1
      \longrightarrow s_0\longrightarrow0)
\] be a nonnegative complex in \(\mathscr T(G)\),
where every \(s_i\) is a set-indexed coproduct of finite permutation
modules, and \(c\)  a bounded complex of finite
permutation modules. If \(s_1\otimes c\simeq0,\) then for all sufficiently large \(n\), we have
\(s_0^{\otimes n}\otimes c\in\operatorname{Loc}^{\otimes}(s).\)
\end{lemma}
\begin{proof}
This is a special case of \cite[Lemma~3.20]{BalmerGallauerGeometry}.
\end{proof}

\begin{proposition}\label{prop7.5}
Let \(P\) be a \(p\)-decomposition group and \(x=\mathcal P_P(H,\mathfrak q)\in X_P\). If \(H\) is open
in \(P\), then \(\mathscr T(P)\) satisfies minimality at \(x\).
\end{proposition}
\begin{proof}
Choose an open normal subgroup
\(V\trianglelefteq_oP\) with \(V\leq H\), and then choose
\(W_0\trianglelefteq_oP\) with \(W_0\leq\Phi(V)\). In the finite group
\(P/W_0\), let \(\operatorname{zul}_{P/W_0}(V/W_0)\) be the compact
object defined in \cite[Notation~7.21]{BalmerGallauerGeometry}, whose support is the image of
modular \((V/W_0)\)-fixed points, and inflate it to a compact object
\(z_V\) of \(\mathscr T(P)\). Its support is exactly
\[
 \operatorname{supp}(z_V)=\operatorname{Im}(\psi_V)
   =\{\mathcal P_P(L,\mathfrak r)\mid V\leq L\}.
\]
Indeed, by \cite[Proposition~7.18]{BalmerGallauerGeometry}, \(y=\mathcal P_P(L,\mathfrak p)\in \operatorname{supp}(z_V)\) if and only if \(V\leq LW_0\). Intersecting with \(V\)
gives \(V=(L\cap V)W_0\), and \(W_0\leq\Phi(V)\) forces \(V\leq L\).

Let \(F=\Psi^V:\mathscr T(P)\to\mathscr T(P/V)\) and let \(R\) be
its right adjoint. Let \(s\) be as in Lemma \ref{lem7.3}.
For every \(K\leq_oP\) with \(V\nleq K\), the compact object
\(k(P/K)\otimes z_V\) has empty support. Indeed, a point in its
support would be of the form \(\mathcal P_P(L,\mathfrak r)\), with
\(L\) satisfying both \(V\leq L\) and
\(L\leq_P K\); normality of \(V\) would then imply \(V\leq K\), a
contradiction. Hence \(k(P/K)\otimes z_V=0\). Since tensor product
preserves coproducts,

\[
 s_1\otimes z_V\simeq
 \bigoplus_{\substack{K\leq_oP\\V\nleq K}}
       (k(P/K)\otimes z_V)\simeq0.
\]
Applying Lemma
\ref{lem7.4} to the bounded compact complex
\(z_V\), and using \(s_0=k\), gives

\[
 z_V\in\operatorname{Loc}^{\otimes}(s)
      =\operatorname{Loc}^{\otimes}(R(1)).
\]
Since \(x\in\operatorname{supp}(z_V)\), it follows from \cite[Theorem~6.3]{BalmerFavi} that
\[
 g_x\in\operatorname{Loc}^{\otimes}(z_V\otimes g_x)
    \subseteq\operatorname{Loc}^{\otimes}(R(1)\otimes g_x).
\]

The spectral map of \(F\) is the closed immersion \(\psi_V\). Since \(x\in \operatorname{Im}(\psi_V)\),  the
fiber over \(x\) is a single point, say \(y\). The target group \(P/V\) is
finite, so \cite[Theorem~9.11]{BalmerGallauerGeometry} makes \(\mathscr T(P/V)\) satisfy minimality at \(y\). The compact spectrum for \(P\) is weakly noetherian by
Proposition~\ref{prop4.5}, and the finite-group spectrum is
noetherian by \cite[Proposition~9.1]{BalmerGallauerGeometry}. Then the conclusion follows from Lemma \ref{lem7.2}.
\end{proof}

\medskip

\begin{proposition}\label{prop7.6}
Let \(P\) be a \(p\)-decomposition group and \(x=\mathcal P_P(H,\mathfrak q)\in X_P\). If \(H=p^rA\) for some \(r\geq0\), then \(\mathscr T(P)\) satisfies minimality at \(x\).
\end{proposition}
\begin{proof}
The group \(B\) is a torsion-free compact \(p\)-adic Lie pro-\(p\) group of dimension \(1\), by
Proposition~\ref{prop5.4} and Proposition~\ref{prop5.6},  we have a finite localization
\[
 \Upsilon_B:\mathscr T(B)\rightleftarrows
 K(\operatorname{Inj}\mathcal A_B):\rho_B.
\]

The smooth module
\(C(B,k)=\operatorname{Coind}_{\{1\}}^B(k)\) is injective,
because restriction to the trivial subgroup is exact and coinduction is
its right adjoint.  We therefore
regard it as an object of \(K(\operatorname{Inj}\mathcal A_B)\)
concentrated in degree zero. Choose a topological generator \(\gamma\) of \(B\). Dualizing the
standard Iwasawa resolution
\[
 0\longrightarrow k[[B]]
   \xrightarrow{\gamma-1}k[[B]]
   \longrightarrow k\longrightarrow0
\]
and using \(k[[B]]^\vee\simeq C(B,k)\), as in
\cite[\S1.1]{HeyerSchneider}, gives an injective resolution
\[
 0\longrightarrow k\longrightarrow C(B,k)
   \xrightarrow{\gamma-1}C(B,k)\longrightarrow0.
\]
Thus  the
exactness of \(\rho_B\) gives
\(
 \rho_B\Upsilon_B(k)
 \in
 \operatorname{thick}\bigl(\rho_B(C(B,k))\bigr).
\) By Proposition~\ref{prop4.4}, applied to \(H=A\) and
\(W_P(A)=B\), the spectrum \(\mathcal V_B\) consists of one
point, denoted by \(y\). Its subgroup parameter is \([1]_B\).
Moreover,
\(\mathscr T(B)=\mathscr T(\mathbb Z_p)\) is stratified by
\cite[Theorem~6.7]{BalmerGallauerArtin}, so the tensor idempotent \(g_y\)
is defined.

By Corollary~\ref{cor5.3},
\(\rho_B\Upsilon_B(k)\) is the right idempotent
corresponding to \(\mathcal V_B=\{y\}\). Hence we have
\(
 g_y\otimes\rho_B\Upsilon_B(k)\simeq g_y,
\)
and thus
\(
 g_y\in
 \operatorname{Loc}^{\otimes}
 \bigl(\rho_B(C(B,k))\otimes g_y\bigr).
\) Let
\(
 I=\operatorname{Infl}_B^P.
\) It is geometric, and the induced spectral map
\[h=\operatorname{Spc}(I^c):X_P\longrightarrow X_B\]
satisfies \(h(x)=y\). Then we have
\(
 I(g_y)\simeq g_{h^{-1}(y)}\) and \(g_x\otimes I(g_y)\simeq g_x.\) Applying \(I\) to \(
 g_y\in
 \operatorname{Loc}^{\otimes}
 \bigl(\rho_B(C(B,k))\otimes g_y\bigr)
\) and tensoring with \(g_x\), we obtain
\[
 g_x\in
 \operatorname{Loc}^{\otimes}
 \bigl(I(\rho_B(C(B,k)))\otimes g_x\bigr).
\]
Now let \(U\) be the right adjoint of \(\operatorname{Res}_A^P\). We claim that \(U(k)\simeq I(\rho_B(C(B,k))).\) The additive functor of \(A\)-coinvariants, applied degreewise to
complexes of permutation modules, induces an exact functor
\[
 (-)_A:\mathscr T(P)\longrightarrow\mathscr T(B)
\]
which is left adjoint to \(I\). In particular,
\(
 k(P/K)_A\simeq k(B/\pi(K)),
\)
where \(\pi:P\to B\) is the quotient map. Consequently, for every
\(t\in\mathscr T(P)\), there are natural isomorphisms
\[
\begin{aligned}
 \operatorname{Hom}_{\mathscr T(P)}
       \bigl(t,I(\rho_B(C(B,k)))\bigr)
 &\simeq
 \operatorname{Hom}_{\mathscr T(B)}
       \bigl(t_A,\rho_B(C(B,k))\bigr)\\
 &\simeq
 \operatorname{Hom}_{K(\operatorname{Inj}\mathcal A_B)}
       \bigl(\Upsilon_B(t_A),C(B,k)\bigr)\\
 &\simeq
 \operatorname{Hom}_{K(k)}
       \bigl(
          \operatorname{Res}_{\{1\}}^B
          \Upsilon_B(t_A),k
       \bigr)\\
 &\simeq
 \operatorname{Hom}_{\mathscr T(A)}
       \bigl(\operatorname{Res}_A^P(t),k\bigr).
\end{aligned}
\]
For the last isomorphism, choose a permutation-complex representative
\(X\) of \(t\). After forgetting the \(B\)-action, an injective
representative of \(\Upsilon_B(X_A)\) is quasi-isomorphic to \(X_A\)
and hence homotopy equivalent to it, since \(k\) is a field. Then the last isomorphism follows from the
natural isomorphism of mapping complexes
\[
 \operatorname{Hom}_{kA}^{\bullet}
       \bigl(\operatorname{Res}_A^P X,k\bigr)
 \simeq
 \operatorname{Hom}_k^{\bullet}(X_A,k).
\]
This proves the claim.  Then we have
\(
 g_x\in
 \operatorname{Loc}^{\otimes}
 \bigl(U(k)\otimes g_x\bigr).
\)

It remains to show that the fiber of
\[
 \operatorname{Spc}(\operatorname{Res}_A^P):
 X_A
 \longrightarrow
 X_P
\]
over \(x\) is a singleton. Let
\(\mathcal P_A(K,\mathfrak s)\) lie in this fiber. By
\cite[Remark~7.6]{BalmerGallauerGeometry}, \(K\) and
\(H\) are \(P\)-conjugate, and therefore we have
\(
 K=H=p^rA.
\)
If \(r=0\), the two relevant spectra are singletons. If
\(r>0\), put \(D=A/p^rA\). For the extension
\[
 1\longrightarrow D\longrightarrow P/p^rA
   \longrightarrow B\longrightarrow1,
\]
the Lyndon--Hochschild--Serre spectral sequence has only columns zero
and one, since \(\operatorname{cd}_p(B)=1\), and the \(B\)-action on
\(H^\bullet(D;k)\) is trivial. It follows that
\[
 H^\bullet(P/p^rA;k)\longrightarrow H^\bullet(D;k)
\]
is surjective with kernel square-zero. Thus restriction induces a homeomorphism on homogeneous prime spectra,
and the fiber over \(x\) consists of a single point. Since
\(\mathscr T(A)=\mathscr T(\mathbb Z_p)\) is stratified by
\cite[Theorem~6.7]{BalmerGallauerArtin}, the conclusion follows from
Lemma~\ref{lem7.2} immediately.
\end{proof}

\begin{lemma}\label{lem7.7}
Let \( Q=A\rtimes_{\theta}C\) with \(A\simeq C\simeq\mathbb Z_p\) and \(\theta:C\to\mathbb Z_p^\times\) injective, and let \(1\neq H\leq_oC\).
 Then there exist a cofinal descending sequence \( U_n\trianglelefteq_oQ \) and an integer \(r\geq0\) with the following properties:
 \begin{enumerate}
 \item[(1)]
 Set \(G_n=Q/U_n, H_n=HU_n/U_n, C_n=C/(C\cap U_n), N_n=N_{G_n}(H_n). \)
 Then for every \(n\),  the inclusion \(C\hookrightarrow Q\) and the projection \(Q\to C\) induce homomorphisms \( j_n:C_n\longrightarrow N_n, \) \(\varpi_n:N_n\longrightarrow C_n \) satisfying \( \varpi_nj_n=\operatorname{id}_{C_n}. \)
 \item[(2)]
 There are \(\alpha_n:N_{n+1}\longrightarrow N_n\) and \(\beta_n:C_{n+1}\longrightarrow C_n \) such that
 \[ \varpi_n\alpha_n=\beta_n\varpi_{n+1}, \qquad \alpha_nj_{n+1}=j_n\beta_n,\qquad \alpha_{n,n+r} = j_n\beta_{n,n+r}\varpi_{n+r}. \]
\end{enumerate}
\end{lemma}

\begin{proof}
Write \(H=p^dC\) and choose a topological generator \(h\) of \(H\). Let \(
 e=v_p\bigl(\theta(h)-1\bigr).
\)
Since \(\theta\) is injective and \(h\neq0\), one has
\(\theta(h)\neq1\), so \(e<\infty\).

Choose a topological generator \(c\) of \(C\) and put \(u=\theta(c)\).
If \(p\) is odd, then \(u\in1+p\mathbb Z_p\), and the standard
valuation formula gives
\(
 v_p\bigl(u^{p^m}-1\bigr)=v_p(u-1)+m.
\)
If \(p=2\), then for \(m\geq1\), the corresponding formula gives
\(
 v_2\bigl(u^{2^m}-1\bigr)
 =
 v_2(u^2-1)+m-1
 \geq m.
\)
It follows that for all sufficiently large \(m\), we have
\(
 \theta(p^mC)\subseteq1+p^m\mathbb Z_p.
\)

Choose \(n_0>\max\{d,e\}\) such that the preceding
inclusion holds for every \(m\geq n_0\). For \(n\geq0\), set
\[
 s_n=n+n_0,
 \qquad
 U_n=p^{s_n}A\rtimes p^{s_n}C.
\]
Then
\(
 \bigl(1-\theta(p^{s_n}C)\bigr)A\subseteq p^{s_n}A.
\)
Since \(p^{s_n}A\) is invariant under \(C\), this implies
\(
 U_n\trianglelefteq_oQ.
\)
Moreover, \(U_{n+1}\leq U_n\), and the subgroups \(U_n\) form a
neighbourhood basis of the identity. Hence they form a cofinal
descending sequence of open normal subgroups of \(Q\).

Since \(s_n>d\), one has
\(
 U_n\cap H=p^{s_n}C\leq p^{d+1}C=pH=\Phi(H).
\)
Furthermore, we have
\(
 C_n=C/p^{s_n}C\) and \(H_n=H/p^{s_n}C\),
so
\(C_n/H_n\simeq C/H.\)
Let \(\pi_C:Q\to C\) denote the projection, then
\(
 \pi_C(U_n)=p^{s_n}C=C\cap U_n.
\)
Thus both the inclusion \(C\hookrightarrow Q\) and the projection
\(\pi_C\) descend to the finite quotients.

We next compute \(N_n=N_{G_n}(H_n)\). Let \(h_n\) denote the image of \(h\) in \(H_n\). For
\(
 (a,c')\in G_n=(A/p^{s_n}A)\rtimes(C/p^{s_n}C),
\)
we have
\[
 (a,c')(0,h_n)(a,c')^{-1}
 =
 \bigl((1-\theta(h))a,h_n\bigr).
\]
Since \(H_n\to C_n\) is injective and \(h_n\) generates \(H_n\), the element
\((a,c')\) normalizes \(H_n\) if and only if it fixes \(h_n\), which is
equivalent to
\(
 \bigl(\theta(h)-1\bigr)a\in p^{s_n}A.
\)
Since \(s_n>e\) and \(v_p(\theta(h)-1)=e\), this is equivalent to
\(
 a\in p^{s_n-e}A/p^{s_n}A,
\)
and hence
\(
 N_n
 =
 \bigl(p^{s_n-e}A/p^{s_n}A\bigr)\rtimes C_n.
\)

Now we can define
\[
 j_n:C_n\longrightarrow N_n,
 \qquad
 c'\longmapsto(0,c'),
\]
and
\[
 \varpi_n:N_n\longrightarrow C_n,
 \qquad
 (a,c')\longmapsto c'.
\]
These are precisely the maps induced by the inclusion \(C\hookrightarrow
Q\) and the projection \(Q\to C\), and they satisfy
\(
 \varpi_nj_n=\operatorname{id}_{C_n}.
\)

Since \(U_{n+1}\leq U_n\), there is a quotient map
\(
 \delta_n:G_{n+1}\longrightarrow G_n
\)
which sends \(H_{n+1}\) onto \(H_n\), and therefore sends
\(N_{n+1}\) into \(N_n\). Consequently, it induces
\[
 \alpha_n=\delta_n|_{N_{n+1}}:
 N_{n+1}\longrightarrow N_n.
\]
The restriction of \(\alpha_n\)  to \(C_{n+1}\) is precisely the quotient map
\(
 \beta_n:C_{n+1}\longrightarrow C_n.
\)
By definition, we have
\(
 \varpi_n\alpha_n=\beta_n\varpi_{n+1}\) and
\(\alpha_nj_{n+1}=j_n\beta_n.
\)

Finally, put \(r=e\). The \(A\)-component of \(N_{n+r}\) is
\(p^{s_n}A/p^{s_n+e}A.\)
Its image under \(\alpha_{n,n+r}\) is zero. Hence \(\alpha_{n,n+r}\) factors through
\(\varpi_{n+r}\).
On \(C_{n+r}\), the resulting map is
\(\beta_{n,n+r}\). Therefore we have
\(\alpha_{n,n+r}=j_n\beta_{n,n+r}\varpi_{n+r}.\)
\end{proof}

\begin{lemma}\label{lem7.8}
Let \((\mathcal A_n,F_n)\) and \((\mathcal B_n,G_n)\) be strict
sequential diagrams of essentially small idempotent-complete
triangulated categories and exact functors. Suppose that exact functors
\(
 R_n:\mathcal A_n\longrightarrow\mathcal B_n,\) \(
 I_n:\mathcal B_n\longrightarrow\mathcal A_n
\)
satisfy
\(
 R_{n+1}F_n\simeq G_nR_n,\)  \(F_nI_n\simeq I_{n+1}G_n.\)
Assume further that there are transition-compatible natural isomorphisms
\(
 R_nI_n\simeq\operatorname{id}_{\mathcal B_n}
\)
and \(
 F_{n,n+r}\simeq I_{n+r}G_{n,n+r}R_n
\) , for some fixed \(r\geq0\).
Then \((R_n)_n\)  induces an exact equivalence
\[
 \operatorname*{colim}_n\mathcal A_n
 \simeq
 \operatorname*{colim}_n\mathcal B_n.
\]
\end{lemma}

\begin{proof}
The compatibility with the transition functors induces exact functors
\[ R:
 \operatorname*{colim}_n\mathcal A_n
 \longrightarrow
 \operatorname*{colim}_n\mathcal B_n,
 \qquad I:
 \operatorname*{colim}_n\mathcal B_n
 \longrightarrow
 \operatorname*{colim}_n\mathcal A_n.
\]
and \(R_nI_n\simeq1\) gives
\( R\circ I\simeq1\). Let \(\iota_n^{\mathcal A}\) and \(\iota_n^{\mathcal B}\) be the
canonical cocone functors. For every \(n\), we have
\[
\begin{aligned}
 \iota_n^{\mathcal A}
 &\simeq \iota_{n+r}^{\mathcal A}F_{n,n+r}\\
 &\simeq \iota_{n+r}^{\mathcal A}
          I_{n+r}G_{n,n+r}R_n\\
 &\simeq  I\, R\,\iota_n^{\mathcal A}.
\end{aligned}
\]
Compatibility with the transition functors therefore yields
\(1\simeq I R\). Hence \( R\) and
\( I\) are inverse equivalences.
\end{proof}

\begin{proposition}\label{prop7.9}
Let \(P=A\rtimes_\chi B\) be a \(p\)-decomposition group and
\(H\leq P\) a nonopen subgroup not contained in \(A\). Then for
\(
 C:=N_P(H),\)
 \(x=\mathcal P_P(H,\mathfrak q)\) and
 \(y=\mathcal P_C(H,\mathfrak q),
\) there is a tensor-triangulated equivalence
\[
 \Phi_H:
 \mathcal L_P(H;x)
 \xrightarrow{\sim}
 \mathcal L_C(H;y)
\]
such that
\(
 \Phi_H\bigl(q_x(g_x)\bigr)\simeq q_y(g_y)
\) and
\(
 \Phi_H=\operatorname{Res}_C^{AC}\circ\rho_{P,AC}^{H,x}.
\)
\end{proposition}

\begin{proof}
By Lemma~\ref{lem4.3},
\(C\simeq\mathbb Z_p\), and there is an integer \(d\geq0\) such that
\(H=p^dC\). Since \(\pi(H)\) is a nonzero closed subgroup of
\(B\simeq\mathbb Z_p\), the subgroup \(\pi(C)\) is open in \(B\), so
\(
 AC=\pi^{-1}(\pi(C))\leq_oP.
\)
Moreover, \(H\cap A=1\) and \(H=p^dC\) imply \(C\cap A=1\). Thus we have
\(
 AC=A\rtimes_{\chi_C}C\)
where \(
 \chi_C:=\chi\circ\pi|_C:C\longrightarrow\mathbb Z_p^\times.
\)
Both \(\pi|_C\) and \(\chi\) are injective, hence so is \(\chi_C\).
In particular, \(AC\) satisfies the hypotheses of
Lemma~\ref{lem7.7}.

Since \(N_P(H)=C\leq AC\),
Theorem~\ref{thm6.16} gives a tensor triangulated
equivalence \[\rho_{P,AC}^{H,x}:\mathcal L_P(H;x)
 \xrightarrow{\sim}
 \mathcal L_{AC}(H;x_{AC}) .\] Let
\(
 f_{AC}:X_{AC}\longrightarrow
        X_P
\)
be induced by restriction. Then we have
\(
 f_{AC}(x_{AC})=x\) and \(
 f_{AC}^{-1}(\{x\})\cap Z_{AC}(H;x_{AC})=\{x_{AC}\}.
\)
Thus
\(
 \rho_{P,AC}^{H,x}\bigl(q_x(g_x)\bigr)
 \simeq q_{x_{AC}}(g_{x_{AC}}).
\)

It remains to prove that the  restriction functor
\[
 \operatorname{Res}_C^{AC}:
 \mathscr T(AC)\longrightarrow\mathscr T(C)
\]
induces a geometric equivalence
\(
 \mathcal L_{AC}(H;x_{AC})
 \xrightarrow{\sim}
 \mathcal L_C(H;y).
\)
 Choose a compact object
\(c_H^{AC}\in\mathscr K(AC)\), inflated from \(AC/N_0\) as in
Lemma~\ref{lem6.3}. Applying Lemma~\ref{lem7.7} and discarding finitely many \(U_n\),
 we may assume that
\(
 U_n\leq N_0
\)
for every \(n\), since \((U_n)_n\) is cofinal and descending and
\(N_0\trianglelefteq_o AC\).
Retain the notation
\[
 G_n=AC/U_n,
 \quad H_n=HU_n/U_n,
 \quad C_n=C/p^nC,
 \quad N_n=N_{G_n}(H_n)
\]
and the maps \(\alpha_n,\beta_n,j_n,\varpi_n\) from that
lemma and let
\(
 \delta_n:G_{n+1}\twoheadrightarrow G_n
\)
be the quotient map.

Let \(\bar x_n\) and \(y_n\) be the images (via inflations) of \(x_{AC}\)
and \(y\), respectively, and set
\(
 z_n=\operatorname{Spc}(j_n^*)(y_n).
\)
Since \(N_n=N_{G_n}(H_n)\) and \(C_n/H_n\simeq C/H\), we have
\[
 \operatorname{Spc}\!\left(\operatorname{Res}_{N_n}^{G_n}\right)
       (z_n)=\bar x_n.
\]
The points are compatible with the induced maps:
\[
 \operatorname{Spc}(\delta_n^*)(\bar x_{n+1})=\bar x_n,
 \qquad
 \operatorname{Spc}(\beta_n^*)(y_{n+1})=y_n,
 \qquad
 \operatorname{Spc}(\alpha_n^*)(z_{n+1})=z_n.
\]
Since \(U_n\leq N_0\), the object \(c_H^{AC}\) has compatible
 representatives in \(\mathscr K(G_n)\). By
Remark~\ref{rem6.8}, their images in the \(\bar x_n\)-stalk have
support \(Z_{G_n}(H_n;\bar x_n)\), and hence compactly generate
\(\mathcal L_{G_n}(H_n;\bar x_n)\).
By \cite[Remark~7.6]{BalmerGallauerGeometry}, we have the following:
\[
\begin{aligned}
 &\operatorname{Spc}(\delta_n^*)^{-1}
       \bigl(Z_{G_n}(H_n;\bar x_n)\bigr)
   \cap \operatorname{gen}_{X_{G_{n+1}}}(\bar x_{n+1})
   =Z_{G_{n+1}}(H_{n+1};\bar x_{n+1}),\\
 &\operatorname{Spc}(\alpha_n^*)^{-1}
       \bigl(Z_{N_n}(H_n;z_n)\bigr)
   \cap \operatorname{gen}_{X_{N_{n+1}}}(z_{n+1})
   =Z_{N_{n+1}}(H_{n+1};z_{n+1}),\\
 &\operatorname{Spc}(\beta_n^*)^{-1}
       \bigl(Z_{C_n}(H_n;y_n)\bigr)
   \cap \operatorname{gen}_{X_{C_{n+1}}}(y_{n+1})
   =Z_{C_{n+1}}(H_{n+1};y_{n+1}),\\
 &\operatorname{Spc}(j_n^*)^{-1}
       \bigl(Z_{N_n}(H_n;z_n)\bigr)
   \cap \operatorname{gen}_{X_{C_n}}(y_n)
   =Z_{C_n}(H_n;y_n),\\
 &\operatorname{Spc}(\varpi_n^*)^{-1}
       \bigl(Z_{C_n}(H_n;y_n)\bigr)
   \cap \operatorname{gen}_{X_{N_n}}(z_n)
   =Z_{N_n}(H_n;z_n).
\end{aligned}
\]
By Proposition~\ref{prop6.13}, we have a
tensor-triangulated equivalence
\[
 R_{G_n,N_n}^{H_n,\bar x_n}:
 \mathcal L_{G_n}(H_n;\bar x_n)
 \xrightarrow{\sim}
 \mathcal L_{N_n}(H_n;z_n).
\]
The identity
\(
 R_{G_{n+1},N_{n+1}}^{H_{n+1},\bar x_{n+1}}
 \delta_n^*
 \simeq
 \alpha_n^*
 R_{G_n,N_n}^{H_n,\bar x_n}
\)
and the support identities above make these equivalences a
compatible morphism of the two filtered diagrams.

On the other hand, we have
\[
\begin{aligned}
 j_{n+1}^*\alpha_n^*&\simeq\beta_n^*j_n^*,&
 \alpha_n^*\varpi_n^*&\simeq\varpi_{n+1}^*\beta_n^*,\\
 j_n^*\varpi_n^*&\simeq1,&
 \alpha_{n,n+r}^*&\simeq
 \varpi_{n+r}^*\beta_{n,n+r}^*j_n^*.
\end{aligned}
\]
Then Lemma~\ref{lem7.8}, together with the integer \(r\) supplied by Lemma~\ref{lem7.7}, implies
\[
 \operatorname*{colim}_n(\mathcal L_{N_n}(H_n;z_n)^c)
 \simeq
\operatorname*{colim}_n(\mathcal L_{C_n}(H_n;y_n)^c).
\]
By Proposition~\ref{prop6.7},
we have
\[
 \mathcal L_{AC}(H;x_{AC})^c
 \simeq
\operatorname*{colim}_n(\mathcal L_{G_n}(H_n;\bar x_n)^c)
 \simeq
\operatorname*{colim}_n(\mathcal L_{N_n}(H_n;z_n)^c).\]
On the other hand,  \(N_C(H)/H=N_{AC}(H)/H=C/H\) shows
\(
 f_C(y)=x_{AC},
\)
 and
\[
\operatorname{supp}_{X_C}
       \bigl(\operatorname{Res}_C^{AC}(c_H^{AC})\bigr)
       \cap\operatorname{gen}_{X_C}(y)=
f_C^{-1}\bigl(\operatorname{supp}_{X_{AC}}(c_H^{AC})\bigr)
       \cap\operatorname{gen}_{X_C}(y)
 =Z_C(H;y),
\]where \(f_C:X_C\to X_{AC}\) is induced by restriction.  For the last equality, a generalization
\(\mathcal P_C(K,\mathfrak s)\) of \(y\)
belongs to the inverse image precisely when
\([K]_{AC}=[H]_{AC}\). Then \(K\) and \(H\) are \(AC\)-conjugate; applying
the projection \(AC=A\rtimes C\twoheadrightarrow C\), which is the
identity on \(C\), gives \(K=H\). The converse is immediate.
Hence
\(
 \overline{\operatorname{Res}}_C^{AC}\bigl(e_{AC}(H;x_{AC})\bigr)
 \simeq e_C(H;y).
\)
Here the bar denotes the functor induced on the corresponding stalks.
Thus restriction induces a tensor-triangulated functor from \(
 \mathcal L_{AC}(H;x_{AC})
\) to \(\mathcal L_C(H;y)\)
, by which we still denote \(\operatorname{Res}_C^{AC}\).
Since
\(
 C\simeq\varprojlim_n C_n
\)
and the  maps
\(\beta_n:C_{n+1}\twoheadrightarrow C_n\) are surjective,
Proposition~\ref{prop6.7} and
Remark~\ref{rem6.8} imply that
\[
 \operatorname*{colim}_n
 \mathcal L_{C_n}(H_n;y_n)^c
 \xrightarrow{\sim}
 \mathcal L_C(H;y)^c.
\]
The strictly commutative square formed by the quotient maps and the
inclusions \(C\leq AC\) and \(C_n\leq G_n\) gives natural
isomorphisms
\[
 \operatorname{Res}_C^{AC}\operatorname{Infl}_{G_n}^{AC}
 \simeq
 \operatorname{Infl}_{C_n}^{C}
 \operatorname{Res}_{C_n}^{G_n}.
\]
These isomorphisms are compatible with the transition maps. Hence,
under the equivalences
\[
 \mathcal L_{AC}(H;x_{AC})^c
 \simeq
 \operatorname*{colim}_n
 \mathcal L_{G_n}(H_n;\bar x_n)^c
\]
and
\[
 \operatorname*{colim}_n
 \mathcal L_{C_n}(H_n;y_n)^c
 \simeq
 \mathcal L_C(H;y)^c,
\]
the functor induced by
\[
 \left(
 j_n^*\circ R_{G_n,N_n}^{H_n,\bar x_n}
 \right)^c:
 \mathcal L_{G_n}(H_n;\bar x_n)^c
 \longrightarrow
 \mathcal L_{C_n}(H_n;y_n)^c
\]
identifies with
\[
 \left(\operatorname{Res}_C^{AC}\right)^c:
 \mathcal L_{AC}(H;x_{AC})^c
 \longrightarrow
 \mathcal L_C(H;y)^c.
\]
The former is an equivalence by the preceding equivalences.
Therefore \(\left(\operatorname{Res}_C^{AC}\right)^c\) is an
equivalence, which implies the desired (tensor-triangulated) equivalence.

Finally, the fiber of \(f_C:X_C\to X_{AC}\) over \(x_{AC}\) is the singleton
\(\{y\}\). Indeed, if \(\mathcal P_C(K,\mathfrak s)\) maps to \(x_{AC}\),
then \(K\) is \(AC\)-conjugate to \(H\). Projection to \(C\) gives \(K=H\), and
so \(\mathfrak s=\mathfrak q\). Consequently we have
\(
 \operatorname{Res}_C^{AC}\bigl(q_{x_{AC}}(g_{x_{AC}})\bigr)
 \simeq q_y(g_y).
\)
This concludes the proof.
\end{proof}

\begin{proposition}\label{prop7.10}
Let \(P\) be a \(p\)-decomposition group and
\(x=\mathcal P_P(H,\mathfrak q)\in X_P\). If \(H\) is nonopen and
\(H\not\leq A\), then \(\mathscr T(P)\) satisfies minimality at
\(x\).
\end{proposition}

\begin{proof}
Let
\(
 C=N_P(H)\)
and \(
 y=\mathcal P_C(H,\mathfrak q).
\)
Since \(C\simeq\mathbb Z_p\),
\cite[Theorem~6.7]{BalmerGallauerArtin} shows that
\(\mathscr T(C)\) is stratified. In particular, \(y\) is weakly
visible and \(\mathscr T(C)\) satisfies minimality at \(y\).
Corollary~\ref{cor6.5} therefore implies
that
\(
 \operatorname{Loc}^{\otimes}_{\mathcal L_C(H;y)}
       \bigl(q_y(g_y)\bigr)
\)
is minimal in \(\mathcal L_C(H;y)\). Then the conclusion follows from
Proposition~\ref{prop7.9} and
Corollary~\ref{cor6.5} immediately.
\end{proof}

\begin{theorem}
\label{thm7.11}
For every \(p\)-decomposition group \(P\), the category
\(\mathscr T(P)\) is stratified.
\end{theorem}
\begin{proof}
Its compact spectrum is weakly noetherian by Proposition
\ref{prop4.5}, and Theorem
\ref{thm4.7} proves local-to-global. It follows from  Lemma~\ref{lem4.3}, Proposition~\ref{prop7.1},
Proposition~\ref{prop7.5},
Proposition~\ref{prop7.6}, and
Proposition~\ref{prop7.10} that  \(\mathscr T(P)\) satisfies minimality at every prime ideal. Thus \(\mathscr T(P)\) is stratified.
\end{proof}

\section{Artin motives over local fields}\label{sec:descent}

In this section we pass from \(p\)-decomposition groups to absolute Galois groups of local fields. When the residue characteristic differs
from \(p\), we transfer stratification from a \(p\)-Sylow subgroup to \(G_F\) using prime-to-\(p\) reduction and finite descent. Combined
with the wild obstruction of Section~\ref{sec:wild}, this classifies the local fields \(F\) for which \(\operatorname{DAM}(F;k)\) is
stratified. Finally, we deduce the telescope conjecture in the tame case.

\begin{lemma}\label{lem8.1}
Let a finite group \(D\) act by spectral automorphisms on a weakly
noetherian spectral space \(X\). Assume that the topological orbit
space \(X/D\) is spectral and the orbit map is spectral. Then
\(X/D\) is weakly noetherian and every  fiber is discrete.
\end{lemma}
\begin{proof}
Write \(\pi:X\to X/D\) for the orbit map.
By the assumptions, quasi-compact open subsets of \(X/D\) correspond exactly to
\(D\)-invariant quasi-compact open subsets of \(X\). By taking complements, there is
a canonical homeomorphism \((X/D)^*\simeq X^*/D.\)

Fix a \(D\)-orbit \(O\). For distinct \(x,y\in O\),  in  \(X^*\)  one has
\(
 x\notin\overline{\{y\}}\) and
 \(y\notin\overline{\{x\}}.\)
Indeed, writing \(y=dx\) and \(d^r=1\), the assumption
\(x\in\overline{\{dx\}}\) would imply
\[
 x\in\overline{\{dx\}},\quad
 dx\in\overline{\{d^2x\}},\quad\ldots,\quad
 d^{r-1}x\in\overline{\{x\}}.
\]
Hence \(dx\in\overline{\{x\}}\), so
\(\overline{\{x\}}=\overline{\{dx\}}\). Since \(X^*\) is \(T_0\),
this forces \(x=dx=y\), a contradiction. For each \(x\in O\), the local
closedness of \(\{x\}\) in \(X^*\) makes \(\overline{\{x\}}\setminus\{x\}
\) closed in \(X^*\). Hence
\(
 B:=\bigcup_{x\in O}\overline{\{x\}}\setminus\{x\}\) and
 \(C:=\bigcup_{x\in O}\overline{\{x\}}\)
are invariant closed subsets, and so \(O=C\setminus B\) is invariant locally
closed. Write
\(
 \pi^*:X^*\longrightarrow X^*/D\simeq (X/D)^*
\)
for the orbit map on the Hochster dual. Since
\(
 (\pi^*)^{-1}\pi^*(B)=B,\)
 \((\pi^*)^{-1}\pi^*(C)=C,
\)
we have
\(
 \pi^*(C\setminus B)=\pi^*(C)\setminus\pi^*(B)
\). The orbit map is closed, so \(\pi^*(B)\) and \(\pi^*(C)\) are closed,
and thus \(\pi^*(O)=\pi^*(C)\setminus\pi^*(B)\) is a locally closed singleton in
\(X^*/D\). This proves that \(X/D\) is weakly noetherian.

The same argument in the original topology shows that, for distinct
points \(x,y\) in the same \(D\)-orbit,
\(x\notin\overline{\{y\}}.\)
Thus there exists an open subset
\(U_{x,y}\subseteq X\) such that
\(
 x\in U_{x,y},\)
\(
 y\notin U_{x,y}.
\)
The subset
\(
 U_x:=
 \bigcap_{y\in D\cdot x\setminus\{x\}}U_{x,y}
\)
is open in \(X\), and it satisfies
\(
 U_x\cap(D\cdot x)=\{x\}.
\)
Thus every point of \(D\cdot x=\pi^{-1}(\pi(x))\) is open in the
subspace topology, so every fiber of \(\pi\) is discrete.
\end{proof}

\begin{lemma}\label{lem8.2}
Let a finite group \(D\) act continuously on a generically noetherian
space \(X\). Then \(X/D\) is generically noetherian.
\end{lemma}
\begin{proof}
Let \(\pi:X\to X/D\) be the closed orbit map. For every \(x\in X\), one has
\[
 \operatorname{gen}_{X/D}(\pi(x))
   =\pi\!\left(\bigcup_{d\in D}
          \operatorname{gen}_X(d x)\right).
\]
Indeed, \(\pi(z)\) generalizes \(\pi(x)\) if and only if
\[
 \pi(x)\in\overline{\{\pi(z)\}}
 \quad\Longleftrightarrow\quad
 x\in\pi^{-1}\!\left(\overline{\{\pi(z)\}}\right)
   =\bigcup_{d\in D}\overline{\{d z\}},
\]
which, after replacing \(d\) by \(d^{-1}\), is equivalent to
\(z\in\operatorname{gen}_X(d x)\) for some \(d\). Thus \(\operatorname{gen}_{X/D}(\pi(x))\) is noetherian for every \(x\).
\end{proof}

\begin{proposition}\label{prop8.3}
Let \(F\) be a nonarchimedean local field with residue characteristic
\(\ell\neq p\). Then \(\mathscr T(G_F)\) is stratified.
\end{proposition}
\begin{proof}
Put \(q=|\kappa_F|\), and write \(I_F^{\mathrm w}\) for the wild
inertia subgroup. By the standard presentation of the maximal tame
quotient
\cite[Theorem~7.5.3]{NeukirchSchmidtWingberg}, there are an arithmetic
Frobenius lift \(\varphi\) and a topological generator \(\tau\) of tame
inertia such that
\[
 G_F/I_F^{\mathrm w}
 \simeq
 \left(\prod_{r\neq\ell}\mathbb Z_r(1)\right)
 \rtimes\widehat{\mathbb Z},
 \qquad
 \varphi\tau\varphi^{-1}=\tau^q.
\]
Write
\(
 \prod_{r\neq\ell}\mathbb Z_r(1)=A\times T',\)
where  \(A=\mathbb Z_p(1)\simeq\mathbb Z_p,\)
 \(T'=\prod_{\substack{r\neq\ell\\r\neq p}}\mathbb Z_r(1).
\) The conjugation action of the unramified quotient on the
\(p\)-primary tame inertia defines a continuous homomorphism
\[
 \theta:\widehat{\mathbb Z}
 \longrightarrow
 \operatorname{Aut}_{\mathrm{cts}}(A)
 \simeq\mathbb Z_p^\times,
 \qquad
 \theta(1)=q.
\]
Let
\(
 m=\operatorname{ord}_{\mathbb F_p^\times}(q\bmod p).
\)
Then \(m\) is prime to \(p\), with \(m=1\) when \(p=2\), and
\(
 \theta(m\widehat{\mathbb Z})
 =\overline{\langle q^m\rangle}
\)
is a pro-\(p\) subgroup of \(\mathbb Z_p^\times\). Indeed, for odd
\(p\) one has \(q^m\in1+p\mathbb Z_p,\)
while \(\mathbb Z_2^\times\) is itself pro-\(2\). Decompose
\(
 m\widehat{\mathbb Z}=B\times C \) where \(B\simeq\mathbb Z_p\) and \(C\) has pro-order prime to \(p\). Since every continuous
homomorphism from a pro-\(p'\) group to a pro-\(p\) group is trivial,
\(C\) acts trivially on \(A\). Hence the inverse image of
\(m\widehat{\mathbb Z}\) in the maximal tame quotient has the form
\[
 \overline G_0
 \simeq
 (T'\rtimes C)\rtimes(A\rtimes B),
\] and the action
\(
 \theta|_B:B\hookrightarrow\mathbb Z_p^\times
\)
is injective. Let \(b\in B\) be the \(B\)-component of
\(m\in m\widehat{\mathbb Z}=B\times C\). Since \(p\nmid m\), the
element \(b\) topologically generates \(B\). Since \(C\) acts
trivially on \(A\), the action of \(b\) on \(A\) is multiplication by
\(q^m\). If the kernel of \(\theta|_B\) is nonzero, it
would be open in \(B\simeq\mathbb Z_p\), and hence the image would be
finite. This would make \(q^m\) a root of unity in
\(\mathbb Z_p^\times\), which is impossible. Thus
\(P_0:=A\rtimes B\)
is a \(p\)-decomposition group.

Let \(G_0\trianglelefteq_oG_F\) be the inverse image of
\(\overline G_0\) and put \(K=\ker(G_0\longrightarrow P_0).\)
Then \([G_F:G_0]=m\) and there is an exact sequence
\[
 1\longrightarrow I_F^{\mathrm w}
 \longrightarrow K
 \longrightarrow T'\rtimes C
 \longrightarrow1.
\]
Since \(I_F^{\mathrm w}\) is pro-\(\ell\) and \(\ell\neq p\), the
group \(K\) has pro-order prime to \(p\). The profinite
Schur--Zassenhaus theorem
\cite[Theorem~2.3.15]{RibesZalesskii} gives a complement \(P\) to
\(K\) in \(G_0\). Consequently, we have
\(
 G_0=K\rtimes P\)
where \(P\simeq P_0.\)
Thus \(P\) is a \(p\)-decomposition group and a \(p\)-Sylow subgroup
of \(G_0\). Since \([G_F:G_0]=m\) is prime to \(p\), it is also a
\(p\)-Sylow subgroup of \(G_F\).

By Theorem~\ref{thm7.11},
\(\mathscr T(P)\) is stratified. Since \(G_0=K\rtimes P\), with
\(K\) of pro-order prime to \(p\), restriction to \(P\) induces a
homeomorphism on spectra by
\cite[Example~3.15]{BalmerGallauerArtin}. Therefore
\cite[Corollary~5.5]{BalmerGallauerArtin} implies that
\(\mathscr T(G_0)\) is stratified.

It remains to descend across the finite normal extension
\(G_0\trianglelefteq G_F\). The tt-ring
\(k(G_F/G_0)\) induces an equivalence
\[
 \mathscr T(G_0)
 \simeq
 \operatorname{Mod}_{k(G_F/G_0)}\bigl(\mathscr T(G_F)\bigr),
\]
under which extension of scalars along \(k(G_F/G_0)\) identifies with
\(\operatorname{Res}_{G_0}^{G_F}\). Since \(m=[G_F:G_0]\) is invertible in \(k\), \(k\) is a retract of
\(k(G_F/G_0)\), and therefore
\[
 k\in
 \operatorname{thick}^{\otimes}\bigl(k(G_F/G_0)\bigr).
\]
Thus \(k(G_F/G_0)\) is nil-faithful, in the sense of \cite[Definition~3.16]{BalmerSeparableExtensions}. Moreover, normality
of \(G_0\) gives
\[
 k(G_F/G_0)\otimes k(G_F/G_0)\simeq\prod_{d\in G_F/G_0}k(G_F/G_0),
\]
and under this decomposition, the two maps
\[
 k(G_F/G_0)\rightrightarrows k(G_F/G_0)\otimes k(G_F/G_0),
 \qquad
 a\longmapsto a\otimes1,\quad
 a\longmapsto1\otimes a,
\]
induce the identity functor and the conjugation functor. Hence
\cite[Theorem~3.19]{BalmerSeparableExtensions} identifies
the resulting topological coequalizer with
\[
 X_{G_F}
 \simeq
 X_{G_0}/(G_F/G_0).
\]
Under this homeomorphism, the spectral map induced by
\(\operatorname{Res}_{G_0}^{G_F}\) is the orbit map.

Since \(\mathscr T(G_0)\) is stratified,
\(X_{G_0}\) is weakly noetherian. Then by Lemma~\ref{lem8.1},
\(X_{G_F}\) is weakly noetherian and the spectral map induced by restriction is surjective with
discrete fibers.

Finally,
\(
 \operatorname{Res}_{G_0}^{G_F}:
 \mathscr T(G_F)\longrightarrow\mathscr T(G_0)
\)
is strongly closed in the sense of \cite[Definition~13.27]{BarthelEtAlCosupport}. Indeed, its right adjoint is coinduction, which
agrees with induction because \([G_F:G_0]<\infty\), and
\(
 \operatorname{Coind}_{G_0}^{G_F}k(G_0/L)
 \simeq
 \operatorname{Ind}_{G_0}^{G_F}k(G_0/L)
 \simeq
 k(G_F/L)
\)
for every open subgroup \(L\leq_oG_0\). Thus the right adjoint
preserves compact objects. The quasi-finite descent theorem
\cite[Theorem~17.16]{BarthelEtAlCosupport} now applies and proves that
\(\mathscr T(G_F)\) is stratified.
\end{proof}

\begin{theorem}
\label{thm8.4}
Let \(k\) be a field of characteristic \(p>0\) and let \(F\) be a nonarchimedean local
field with residue field \(\kappa_F\). Then \(\operatorname{DAM}(F;k)\) is stratified if and only if \(\operatorname{char}(\kappa_F)\neq p.\)
\end{theorem}
\begin{proof}
This follows from Proposition \ref{prop8.3} and Proposition \ref{prop3.4}.
\end{proof}

\begin{remark}\label{rem8.5}
For archimedean local fields, \(\operatorname{DAM}(\mathbb C;k)\) and \(\operatorname{DAM}(\mathbb R;k)\) are stratified by \cite[Theorem~9.11]{BalmerGallauerGeometry}. Together with the nonarchimedean case above, this completes the classification of all
local fields \(F\) for which \(\operatorname{DAM}(F;k)\) is stratified.
\end{remark}

\begin{corollary}
\label{cor8.6}
Let \(F\) be a nonarchimedean local field with \(\operatorname{char}(\kappa_F)\neq p.\) Then the telescope conjecture holds for \(\operatorname{DAM}(F;k)\).
\end{corollary}
\begin{proof}
In the proof of Proposition~\ref{prop8.3}, we have \[
X_{G_F}
 \simeq
 X_{G_0}/(G_F/G_0)\simeq
 X_{P}/(G_F/G_0)
\] for a \(p\)-decomposition group. Then by Proposition~\ref{prop4.5} and Lemma~
\ref{lem8.2}, \(
 \operatorname{Spc}(\operatorname{DAM}^{\rm gm}(F;k))\) is generically noetherian. Hence the telescope conjecture holds for \(\operatorname{DAM}(F;k)\) by Theorem~\ref{thm8.4} and \cite[Theorem~9.11]{BarthelHeardSanders}.
\end{proof}

\vspace{4mm}

\textbf{Peng Xu}\\
School of Mathematics, Nanjing University,
Nanjing 210093, P. R. China;\\
E-mail: \textsf{602023210017@smail.nju.edu.cn}\\[1mm]

\end{document}